\documentclass[reqno,11pt,a4paper]{amsart}
\usepackage{amsmath,amssymb,latexsym}

\usepackage{enumerate}
\usepackage[all]{xy}
\usepackage{fancyhdr}
\usepackage{color}
\usepackage{tikz-cd}
\usepackage{tikz}
\usepackage{nameref}
\usepackage{csquotes}
\usetikzlibrary{matrix,arrows,decorations.pathmorphing}

\usepackage{graphics}
\usepackage{hyperref}

\hypersetup{
	colorlinks=true,
	linkcolor=black,
	filecolor=magenta,      
	urlcolor=black,
	pdftitle={Overleaf Example},
	pdfpagemode=FullScreen,
}

\newcommand{\Z}{\mathbb Z}
\newcommand{\Q}{\mbox{$\mathbb Q$}}

\newcommand{\calL}{\mathcal L}

\newcommand{\surj}{\twoheadrightarrow}
\newcommand{\inj}{\hookrightarrow}

\newcommand\ran{{\mathrm{rank}}}

\newcommand\cora{{\mathrm{corank}}}

\newcommand{\K}{K_\infty}
\newcommand{\f}{F_\infty}

\newcommand{\QQ}{\mathbb{Q}}
\newcommand{\N}{\mathbb{N}}

\newcommand{\FF}{\mathbb{F}}

\newcommand{\Hom}{\mathrm{Hom}}

\newcommand{\Ext}{\mathrm{Ext}}

\newcommand{\gal}{\mathrm{Gal}}
\newcommand{\cyc}{\mathrm{cyc}}
\newcommand{\Ker}{\mathrm{Ker}}
\newcommand{\coker}{\mathrm{Coker}}

\newcommand{\corank}{\mathrm{corank}}

\newcommand{\fla}{\operatorname{fl}}

\newcommand{\lra}{\longrightarrow}
\newcommand{\ra}{\rightarrow}
\newcommand{\ilim}{\varinjlim}
\newcommand{\plim}{\varprojlim}

\usepackage[OT2,OT1]{fontenc}
\newcommand\cyr{%
\renewcommand\rmdefault{wncyr}
\renewcommand\sfdefault{wncyss}
\renewcommand\encodingdefault{OT2}
\normalfont\selectfont}

\DeclareTextFontCommand{\textcyr}{\cyr}

\newtheorem{theorem}{Theorem}[section]
\newtheorem{proposition}[theorem]{Proposition}
\newtheorem{lemma}[theorem]{Lemma}
\newtheorem{remark}[theorem]{Remark}
\newtheorem{corollary}[theorem]{Corollary}
\newtheorem{definition}[theorem]{Definition}

\newtheorem{conjecture}{Conjecture}
 
\newtheorem*{theorem*}{Theorem}
\newtheorem{example}{Example}

\numberwithin{theorem}{section}

\theoremstyle{definition}

\theoremstyle{plain}
\newtheorem*{namedthm}{\namedthmname}
\newcounter{namedthm}

\newtheoremstyle{named}{}{}{\itshape}{}{\bfseries}{.}{.5em}{\thmnote{#3's }#1}
\theoremstyle{named}

\newtheorem*{corollary*}{Corollary}

 \DeclareFontFamily{U}{wncy}{}
\DeclareFontShape{U}{wncy}{m}{n}{<->wncyr10}{}
\DeclareSymbolFont{mcy}{U}{wncy}{m}{n}
\DeclareMathSymbol{\Sh}{\mathord}{mcy}{"58} 

\newtheoremstyle{named}{}{}{\itshape}{}{\bfseries}{.}{.5em}{\thmnote{#3's }#1}
\theoremstyle{named}

\newtheorem*{conjecture*}{Conjecture}
\newtheorem{conj}{Conjecture 1}

\makeatletter
\newcommand{\customlabel}[2]{%
	\protected@write \@auxout {}{\string \newlabel {#1}{{#2}{\thepage}{#2}{#1}{}} }%
	\hypertarget{#1}{#2}
}
\makeatother



\hypersetup{
   colorlinks=true, 
    linktoc=all,     
    linkcolor=blue,
    citecolor=blue,}

\usepackage[capitalise,nameinlink]{cleveref}

\newcommand{\Keywords}[1]{\par\noindent
{\small{Keywords and phrases}: #1}}

\newcommand{\AMS}[1]{\par\noindent
{\small{AMS Subject Classification}: #1}}

\author{Sohan Ghosh, Somnath Jha, Sudhanshu Shekhar}
\address{Department of Mathematics and Statistics, IIT Kanpur, Kanpur 208016, India}
\email{gsohan@iitk.ac.in, jhasom@iitk.ac.in, sudhansh@iitk.ac.in }

\begin{document}
\title{Iwasawa theory of fine Selmer groups over global fields}
\begin{abstract}
The $p^\infty$-fine Selmer group of an elliptic curve $E$ over a number field $F$ is a subgroup of the classical $p^\infty$-Selmer group of $E$ over $F$. Fine Selmer group  
is closely related to the 1st and 2nd Iwasawa cohomology groups.  
Coates-Sujatha observed that the  structure of  the fine Selmer group of $E$ over a $p$-adic Lie extension  of a number field is intricately  related to some deep questions in classical Iwasawa theory; for example, Iwasawa's classical $\mu$-invariant vanishing conjecture. In this article, we study the properties of the $p^\infty$-fine Selmer group of an elliptic curve over  certain $p$-adic Lie extensions of a number field. We also define  and discuss $p^\infty$-fine Selmer group of an elliptic curve  over  function fields of characteristic $p$ and also of characteristic $\ell \neq p.$ We relate our study  with a conjecture of Jannsen. 
\end{abstract}	
\maketitle
\let\thefootnote\relax\footnotetext{
\AMS{11R23,  11G05, 11S25, 11R60}
}
\section*{Introduction}
We fix  an odd prime  $p$ throughout.
Let $E$ be an elliptic curve defined over a number field $F$ and let $S(E/F)$ be the $p^\infty$-Selmer group of $E$ over $F$ (Definition \ref{defn1.1}). The $p^\infty$-fine Selmer group $R(E/F)$ is a subgroup of $S(E/F)$ (Definition \ref{defn1.1}), obtained by putting stringent local {conditions} at primes dividing $p$. Let us assume that $E/\Q$ has good, ordinary reduction at  $p$ and $\Q_\cyc$ be the cyclotomic $\Z_p$ extension of $\Q$. Let $u$ be the unique prime in $\Q_\cyc$ dividing $p$ and $\Q_{\cyc,u}$ be the completion at $u$.  It is known 
 that $H^2(G_S(\Q_\cyc), E_{p^\infty})=0$  \cite{ka2}.  Then  the following exact sequence can be deduced from  \S \ref{section2}, equations \eqref{equation3} and \eqref{eq6}: 
\begin{small}{$$0\rightarrow H^1_\mathrm{Iw}(T_pE/\QQ_\cyc) \rightarrow (E(\QQ_{\cyc,u}) \otimes \Q_p/\Z_p)^{\vee} \rightarrow S(E/\Q_\cyc)^{\vee}\rightarrow R(E/\Q_\cyc)^\vee\rightarrow 0.$$}\end{small}
  Using Kato's Euler system \cite{ka2} and the Perrin-Riou-Coleman map, from the above equation, we  have an alternative formulation of the Iwasawa main conjecture involving $R(E/\Q_\cyc)^\vee$ (see  \cite[Theorem 16.6.2]{ka2}). 
The fine Selmer group has been studied in Iwasawa theory throughout, under different names, like $\Sh_1,\text{Sel}_0$ by Billot, Greenberg, Kurihara \cite{ku}, Perrin-Riou \cite{pr} and others over the cyclotomic $\Z_p$ extensions of number fields. Coates-Sujatha \cite{cs}  formally defined { the} fine Selmer group of an  elliptic curve and we largely follow their notation in this article. They observed  an important  relation between the structure of $R(E/F_\cyc)$ and  Iwasawa's $\mu=0$ conjecture  about the growth of the $p$-part of the ideal class group in the  cyclotomic tower for $F_\cyc$.
 
 \cite[{\bf Theorem 3.4}]{cs}\label{thm:iwawasa} {\it Let $E/F$ be an elliptic curve and $p$ be an odd prime such that  $F(E_{p^\infty})/F$ is a pro-$p$ extension. Then $R(E/F_\cyc)^{\vee}$ is a finitely generated $\Z_p$ module if and only if   Iwasawa's $\mu=0$ conjecture holds for $F_\cyc$.}

 Motivated by this,  \cite{cs} formulated the following conjecture:
 \begin{conjecture}\label{conja}
 	For all elliptic curves $E$ over $F$, $R(E/F_\cyc)^{\vee}$ is a finitely generated $\Z_p$ module.
 \end{conjecture}
 Following Conjecture \ref{conja},  the fine Selmer group over { the} cyclotomic $\Z_p$ { extensions} of a number field have been studied  by various authors including \cite{wu, lim2,lm}, Jha-Sujatha and others. 

\medskip

\noindent {\bf Main results of the article: part (a): }{\it In this article, we
initiate the study of fine Selmer group of an elliptic curve over $p$-adic Lie extensions of function fields of characteristic  $p$ as well as  $\ell\neq p$. In one of our main results in this article, we show that the analogues of Conjecture \ref{conja} are true for the $p^\infty$-fine Selmer groups over  function fields of characteristic  $p$ (Theorem \ref{theo:muinvariant} and Corollary \ref{cormu}) and  characteristic $\ell \neq p$  (Remark \ref{mu0lneqp}), respectively.}

 Coates-Sujatha also studied the structure of the fine Selmer over more general  $p$-adic Lie extensions of a number field, instead of the cyclotomic $\Z_p$ extension, in the framework of so-called non-commutative Iwasawa theory \cite{cfksv}.  
Motivated by a conjecture of Greenberg \cite[Conjecture 3.5]{gr4},  the following was predicted regarding the structure of the fine Selmer group over an {\it admissible} (see Definition \ref{admis}) $p$-adic Lie extension of a number field:\begin{conjecture}\cite{cs}\label{conjb}
 Assume that the  Conjecture \ref{conja} holds for $E$ over  $F_\cyc$. Let $F_\infty$ be an admissible $p$-adic Lie extension  of  $F$ such that   $G=\gal(F_\infty/F)$ has dimension at least 2 as a p-adic Lie group. Then $R(E/F_\infty)^{\vee}$ is a pseudonull (defined in \S \ref{section2})  $\Z_p[[G]]$ module.
 \end{conjecture}
Following Conjecture \ref{conjb}, the properties of the fine Selmer group over  $p$-adic Lie extensions of a number field have been investigated  by various authors{,} including \cite{ lp, lp2,jh}, Lim and the third named authors of the article. 
In particular,  Lei and Palvannan in \cite{lp} and \cite{lp2}, have studied the pseudonullity of fine Selmer groups of elliptic curves and Hida families over the $\Z_p^2$ -extension of an imaginary quadratic field $K$, respectively.  The work of Hachimori-Sharifi \cite{hs} is also related to the Conjecture \ref{conjb}.

\medskip

\noindent {\bf Main results of the article: part (b): } {\it In another main result { of} this article, we establish analogues of Conjecture \ref{conjb} (Theorem \ref{pseudonullity} and Theorem \ref{theo:pseudomain3}) over  function fields of characteristic $p$. On the other hand, over  function fields of characteristic $\ell \neq p$, we give an explicit counterexample in Example \ref{example} to show that the analogue of Conjecture \ref{conjb} does not hold in that setting.}

Another theme of this article  is the study of the $G$-Euler characteristic of the fine Selmer group $R(E/F_\infty)$ {(see Definition \ref{defn:euler})}. { Put $\Gamma:=\gal(\Q_\cyc/\Q)\cong \Z_p$.} Recall that the $\Gamma$-Euler characteristic of the Selmer group $S(E/\Q_\cyc)$ encodes important arithmetic properties of the elliptic curve. Let $E/\mathbb{\QQ}$ be an elliptic curve  with good, ordinary reduction at $p$ and assume that $L_E(s)$, the Hasse-Weil $L$-function of $E/\QQ$ satisfies $L_E(1) \neq 0$. Then, under suitable hypotheses, \begin{small}$	\chi(\Gamma,S(E/\QQ_\cyc)):=\dfrac{\# H^0(\Gamma, S(E/\Q_\cyc))}{\# H^1(\Gamma, S(E/\Q_\cyc))}$\end{small} is related to \begin{small}$ {L_E(1)}/{\Omega_E}$\end{small} (cf. 
\cite[Theorem 4.1]{gr3}). Over an admissible $p$-adic Lie extension $F_\infty$ of  $F$, the existence of the $G$-Euler characteristic of the dual Selmer group $S(E/F_\infty)$ and its relation with the special values of the $L$-function of $E$ has been established in various cases  due to Coates-Howson, \cite{css,hv}  etc. On the other hand, Wuthrich \cite{wu} under certain conditions, has proven the existence of $\chi\big(\Gamma,R(E/F_\cyc)\big)$ and has given a formula to compute it. In this paper, we discuss the $G$-Euler characteristic of $R(E/F_\infty)$, where $F_\infty$ is { a} certain admissible $p$-adic Lie extension of a number field or a function field $F$. We would like to mention that it does not seem to be easy to prove the existence of the Euler characteristic even over a specific non-commutative $p$-adic Lie extension of a number field, like the false-Tate curve extension (see Remark \ref{limfalse}).   \\


\noindent {\bf Main results of the article: part (c):} {\it Our main result over number fields is Theorem \ref{mainnumber}. In this theorem, under suitable hypotheses and  assuming Jannsen's conjecture (Conjecture \ref{ja}),  we prove the existence of  the $G$-Euler characteristic of $R(E/F_\infty)$  over the false-Tate curve extension $F_\infty$ of $F$  (see \S \ref{section2}). In fact, in \S \ref{section2}, we  also prove the existence of $\chi\big(G,R(E/F_\infty)\big)$ without assuming Conjecture \ref{ja}, as long as $\Z_p[[H]]$ corank of $R(E/F_\infty)$ is at most 1 (see Propositions  \ref{lemma2.8} and  \ref{thm:euler}). We stress that  Theorem \ref{mainnumber} and Proposition \ref{lemma2.8} are valid irrespective of whether $E$ has ordinary or supersingular  reduction  above $p$. We also prove the existence of $\chi\big(G, R(E/F_\infty)\big)$ over the false-Tate curve extension $F_\infty$ of a function field $F$ of characteristic $\ell \neq p$; in this setting the analogue of Conjecture \ref{ja} has been already proved by Jannsen. The existence of the $G$-Euler characteristic of $p^\infty$-fine Selmer group over $\Z_p^d$ extension a function field of characteristic $p$, under appropriate { conditions}, is established in Proposition \ref{thm:control}.}


In the number field case,  for proving Theorem \ref{mainnumber}, we assume Jannsen's Conjecture and make use of a result of Kato \cite[Theorem 5.1]{ka} along with other Iwasawa theoretic techniques and the theorem holds whether $E$ has  ordinary or supersingular reduction at the primes above $p$. For the other results in \S \ref{section2}, we use the structure of modules over non-commutative Iwasawa algebras and the properties of elliptic curve. 
In \S \ref{section4}, we notice that the image of the kummer map is zero, thus $p^\infty$-fine Selmer group over  function fields of characteristic $ \ell \neq p$ coincides with the $p^\infty$-Selmer group. Over function fields of characteristic $p$, we give emphasis to two special $\Z_p$ extensions; the arithmetic or the unramified $\Z_p$ extension  and { the geometric $\Z_p$ extension constructed  from Carlitz module}, a particular type of Drinfeld module (see \S \ref{subs3.1}).  
In addition, we also provide  modest evidence towards Conjecture \ref{conjb}  over a general $\Z_p^d$ extension in Corollary \ref{theo4.8}. In \S \ref{section5}, our main tool  is to compare the fine Selmer group of $E[p]$ with the corresponding fine Selmer groups of the group schemes $\mu_p$ and $\Z/p\Z$ and then make use of the classical results on the divisor class group over  function field. In the characteristic $p$ setting, we discuss the dependence of the fine Selmer group on the set $S$ (Remark \ref{remark2.2}) and also comment on the zero Selmer group  (Remark \ref{remark0selmer}). 

The structure of the article is the following: \S \ref{section2} contains the results over  number fields. In particular, we discuss various cases in which, over the false-Tate curve extension $F_\infty/F$,   $\chi(G, R(E/F_\infty))$  exists and  Conjecture \ref{conjb} holds. In \S \ref{section4}, we consider results over  function fields of characteristic $\ell\neq p$ and we notice that in this setting, an  analogue of Conjecture \ref{conja} is true but an analogue of Conjecture \ref{conjb} is false. The results for $p^\infty$-fine Selmer group over function field{s} of characteristic $p$ are contained in \S \ref{section5} and  there we prove that an analogue of Conjecture \ref{conja}  holds and moreover show that an analogue of Conjecture \ref{conjb} holds over a certain class of $\Z_p^d$ extensions.

\section{Results over number fields }\label{section2}
Throughout the section \ref{section2}, $E$ will be an elliptic curve defined over a number field $F$ with {\it good reduction at all the primes of $F$ dividing $p$}. (Note: Our Theorem \ref{mainnumber} and Proposition \ref{lemma2.8} { covers both the case when E has ordinary reduction and the case when E has supersingular reduction} at primes above $p$.
However, from Proposition \ref{thm:euler} onwards, we assume that $E$ has ordinary reduction at all  primes above $p$.)  Also, throughout \S \ref{section2}, $S$ will denote a finite set of primes of  $F$ containing the primes dividing $p$, the infinite primes of $F$ and the primes where $E$ has bad reduction. 
Let $F_S$ be the maximal algebraic extension of $F$ unramified outside $S$. For a field $L$ with $F\subset L \subset F_S$, write $G_S(L)=\mathrm{Gal}(F_S/L)$. For an abelian group $M$, $M[p^n]$ and  $M(p)$ will respectively denote its $p^n$-torsion and $p$-primary torsion subgroup of $M$.  For a compact or discrete $\Z_p$ module $M$, let $M^{\vee}:=\Hom_{\text{cont}}(M,\Q_p/\Z_p)$ denote its Pontryagin dual. Let $E_{p^n}$  (respectively  $E_{p^\infty}$)  denote  the Galois module $E(\overline{\Q})[p^n]$  (respectively $E(\overline{\Q})(p)$). Let $T_p(E)=\underset{n}{\varprojlim}\ {E_{p^n}}$  be the Tate module associated to $E$.   Let $F_\cyc$ be the cyclotomic $\Z_p$ extension of  a number field $F$ and set $\Gamma:=\mathrm{Gal}(F_\cyc/F)$. For any  profinite group $\mathcal{G}$, the Iwasawa algebra $\Z_p[[\mathcal G]]$ is defined by \begin{small}$\Lambda({ \mathcal G}):= \Z_p[[\mathcal G]]:=\underset{U}\plim\ \Z_p\ [\mathcal G/U]$\end{small}, where $U$ varies over open normal subgroups of $\mathcal G$ and the inverse limit is taken with respect to the natural projection maps.
\begin{definition}\label{admis}  
A Galois extension $F_\infty/ F$ is called an admissible $p$-adic Lie extension  if  $F_\cyc\subset  F_\infty$, at most finitely many primes of $F$ ramify in $F_\infty$ and   $G:=\gal(F_\infty/F)$ is a compact $p$-adic Lie group  without an element of order $p$.  
\end{definition}

 The study of {  torsion modules and  }pseudonull modules play an important role in commutative Iwasawa theory. { Let us consider the one variable Iwasawa algebra  $\Lambda=\Z_p[[\Gamma]]\cong \Z_p[[T]]$. Given a finitely generated torsion $\Lambda$ module $M$, there is a structure theorem and one can associate a characteristic ideal $( p^{\mu(M)}f_M(T))$, where $\mu(M)\in \Z_{\geq 0}$ and $f_M(T)$ is an irreducible monic Weierstrass polynomial in $\Z_p[T]$. Let $E/\Q$ be an elliptic curve with ordinary reduction at a prime $p$. Then the  main conjecture of Iwasawa theory asserts that  $S(E/\Q_\cyc)^\vee$  is a torsion $\Lambda$ module with the characteristic ideal generated by the $p$-adic $L$-function of $E/\Q$ (see \cite[Conjecture 1.13, Page 59]{gr3}).

The Iwasawa algebra $\Lambda$ is not a PID and the structure theorem for a finitely generated torsion module is valid only up to a pseudo-isomorphism i.e. a $\Lambda$ module homomorphism with pseudonull kernel and cokernel.  Thus pseudonull module naturally appears in Iwasawa theory, although they have trivial characteristic ideal. In arithmetic situation, for an elliptic curve $E/\Q$ as above, Greenberg \cite[Proposition 10]{gr} has shown that, under suitable hypotheses,  $S(E/\Q_\cyc)^\vee$ has no non-zero $\Lambda$ pseudonull submodule. For a non-commutative $p$-adic Lie extension $F_\infty$ of a number field $F$, whether  $S(E/F_\infty)^\vee$ has non-zero  $\Z_p[[\mathrm{Gal}(F_\infty/F)]]$ pseudo-null  submodule, is studied extensively (cf. \cite[Theorem 2.6]{hv}). As mentioned earlier, there is  a  conjecture of Greenberg \cite[Conjecture 3.5]{gr4}) on the pseudonullity of certain Iwasawa module in classical Iwasawa theory which motivated    \cite[Conjecture \ref{conjb}]{cs}. We refer to \cite[Chapter 5, \S1]{nsw} for more details on the definition and motivation of pseudonull modules in commutative Iwasawa theory.}

 For a compact $p$-adic Lie group $G$ without an  element of order $p$, the notion of torsion modules and pseudonull modules over Iwasawa algebra $\Z_p[[G]]$ was generalised by Venjakob. If $M$ is a finitely generated $\Z_p[[G]]$ module, then $M$ is said to be a pseudonull (respectively torsion) $\Z_p[[G]]$ module if
 $\Ext^i_{\Z_p[[G]]} (M,\Z_p[[G]]) = 0$ for $i = 0, 1$ (respectively $i=0$). 
We have the following criterion for pseudonullity due to Venjakob {(see \cite[Lemma 3.1]{hs}, \cite[Proposition 5.4, Example 2.3]{ve2})}: 
 \begin{theorem}[Venjakob]\label{theo:pseudo}
 	Let $\f$ be an admissible $p$-adic Lie extension with Galois group $G$ and set $H:=\mathrm{Gal}(F_\infty/F_\cyc)$. Let $M$ be a $\Z_p[[G]]$ module, which is a finitely generated $\Z_p[[H]]$ module. Then $M$ is a pseudonull $\Z_p[[G]]$ module if and only if it is $\Z_p[[H]]$ torsion. 
 \end{theorem}
 We also recall the following result:
 \begin{theorem}\cite[Theorem 1.1]{ho}\label{theo:Ho}
 	Let $G$ be a compact,  pro-$p$, $p$-adic Lie group without any element of order $p$ and $M$  be a cofinitely generated discrete $\Z_p[[G]]$ module. Then $\cora_{\Z_p[[G]]}(M)=\sum_{i\geq 0} (-1)^i \cora_{\Z_p} (H^i (G,M))$. \qed
 \end{theorem}
Let $F,S$ be as above and  $L$ be a finite extension of $F$ with $L\subset F_S$. For each $v \in S$, set \begin{small}
	$J_v(E/L):=\underset{w\mid v}\prod H^1(L_w,E(\overline{L}_w))(p)$
\end{small} and \begin{small} $K^i_v(E/L):= \underset{w\mid v}\prod H^i(L_w,E_{p^\infty})$\end{small}, for $i=0,1$. Here $\overline{L}_w$ is the completion of $L$ at $w$. 
The $p^\infty$-Selmer and $p^\infty$-fine Selmer group of $E$ over $L$, denoted respectively as $S(E/L)$ and $R(E/L)$ are defined as:
\begin{definition}\label{defn1.1}
\begin{small}
\begin{equation}
S(E/L):=\Ker(H^1(G_S(L),E_{p^{\infty}})\longrightarrow \underset{v\in S}\oplus J_v(E/L)) 		  
	\end{equation}
	\begin{equation}
	R(E/L):=\Ker (H^1(G_S(L),E_{p^{\infty}})\lra \underset{v\in S}\oplus K^1_v(E/L)),
\end{equation}
\end{small}
\end{definition}
In fact, $S(E/L)$ and $R(E/L)$ are independent of $S$ as long as $S$ contains all  primes of bad reduction of $E/K$. We have the following relation between $R(E/L)$ and $S(E/L)$ \cite{cs}
\begin{small}
 \begin{equation}\label{equation3}
 	0\lra R(E/L)\lra S(E/L)\lra \underset{w \mid p}{\oplus} (E(L_w)\otimes \frac{\Q_p}{\Z_p})
 \end{equation}
\end{small}
 The definitions of $S(E/\mathcal L)$ and $R(E/\mathcal L)$, over an infinite extension $\mathcal L$ of $F$ contained in $F_S$, extends as usual,  by taking inductive limit over all finite subextensions of $\mathcal L$ containing $F$. Moreover,  using the Poitou-Tate exact sequence for $E_{p^{\infty}}$ over $L$  and { taking} 
 inductive limits over all such finite extensions  $F \subset L\subset \mathcal{L}$,  \cite[Equations (44), (45)]{cs} we get:
\begin{small}
 \begin{equation}\label{eq5}
 0\lra H^0(\mathcal{L},E_{p^{\infty}})\lra\underset{v\in S}\oplus K^{0}_v(E/\mathcal{L})\lra H^2_{\mathrm{Iw}}(T_pE/ \mathcal{L})^{\vee}\lra R(E/\mathcal{L})\lra 0,
 \end{equation}
\end{small}
\begin{small}
	\begin{equation}\label{eq6}
		\begin{array}{lcl}
			0\ra R(E/\mathcal{L})\ra H^1(G_S(\mathcal{L}),E_{p^{\infty}})\ra \underset{v\in S}\oplus K^{1}_v(E/\mathcal{L})  \ra H^1_{\mathrm{Iw}}(T_pE/ \mathcal{L})^{\vee}\ra  H^2(G_S(\mathcal{L}),E_{p^{\infty}})\ra 0.
		\end{array}
	\end{equation}
\end{small}
Here for $i=1,2,\ H^i_{\mathrm{Iw}}(T_pE/\mathcal{L}):=\underset{\small{L}}{\plim} \ H^i(K_S/L,T_p(E))$   is the Iwasawa cohomology of $T_pE$ over $\mathcal L$ and the projective limit is taken with respect to the corestriction maps. Equations \eqref{eq5} and  \eqref{eq6} { explain} the relation between the Iwasawa cohomologies and the fine Selmer group.

  Let $G$ be a compact $p$-adic Lie group without any element of order $p$. Then, by  results of Brumer and Lazard, the global dimension  of $\Z_p[[G]]$ is finite.   
\begin{definition}\label{defn:euler}
	Let $G$ be a compact $p$-adic Lie group without any $p$-torsion element and $M$ be a discrete $G$-module. If $H^i(G,M)$ is finite for all $i\geq 0$, then we say that the $G$- Euler characteristic of $M$ exists and it is defined as
	\begin{small} $$\chi(G,M)=\underset{i\geq 0}\prod(\# H^i(G,M))^{(-1)^i}. $$
		\end{small}
\end{definition}

In the above setting, if $G$ is  commutative or more generally `finite by nilpotent' \cite[Remark 1.5]{js}, then it is known that  $\chi(G,M)$ exists  $\Leftrightarrow H^0(G,M)$ is finite.

 Now we discuss the Euler characteristic of fine Selmer groups of elliptic curves over perhaps the simplest non-commutative admissible $p$-adic Lie extension, the so called  false-Tate curve extension $F_\infty$ of $\QQ(\mu_p)$. We prove the finiteness of the Euler characteristic of the fine Selmer group of an elliptic curve over the false-Tate curve extension in two different situations. For the first result, we explore the relation between the fine Selmer group and Conjecture {\ref{ja}} of Jannsen on the  twist of $\ell$-adic cohomology of a motive. Using this together with a result of Kato \cite{ka}, we prove the finiteness of the Euler characteristic of the fine Selmer group. For the second {set} of results, we do not assume any conjecture but instead make some hypothesis on the corank of the Selmer group of $E$ over $F_\cyc$.

  Let $\mu_{p^n}$ denote the group of $p^n$-th roots of unity in $\bar{\Q}$. 
  { Let $\chi_p:G_\QQ\lra \Z_p^\times$ be the $p$-adic cyclotomic character.} Let $F$ be a number field containing $\mu_p$. Now, let $m$ be  a $p$-power free positive integer. Then $\f:=\underset{n}\bigcup F(\mu_{p^n}, m^{1/p^n})$ is the  false Tate curve extension over $F$. {  Put $G:=\mathrm{Gal} (\f/F)$, $H:=\gal(F_\infty/F_\cyc) \cong \Z_p$ and $\Gamma:=\gal(F_\cyc/F)\cong G/H \cong \Z_p$, where the action of $\Gamma$ on $H$ is given by the $p$-adic cyclotomic character. More precisely, fix  topological generators $\gamma$ of $\Gamma$ and $h$ of $H$, respectively and let $\tilde{\gamma}$ be a lift of $\gamma$ in $G$. Then   the action of $\tilde{\gamma}$ on $H$ is given by $\tilde{\gamma}\cdot h= \tilde{\gamma}h\tilde{\gamma}^{-1}=h^{\chi_p(\gamma)}$. Hence $G\cong H\rtimes \Gamma\cong \Z_p\rtimes \Z_p$.  }

  Given a $G_\QQ$ module $M$,  we denote by $M(k)$ the twist of $M$ by $\chi_p^k$.  
  Define
  \begin{small}
 $$
 	R(E^k/F):=\Ker (H^1(G_S(F),E_{p^{\infty}}(k))\lra \underset{v\in S}\bigoplus\ \underset{w\mid v}\prod H^1(F_w,E_{p^\infty}(k))).
 $$
 \end{small}

Next, we discuss the finiteness of $H^0(G,R(E/\f))=R(E/\f)^G$. 
\begin{proposition}\label{thm:H_0}
	Assume that $R(E/F)$ is finite. Then $R(E/\f)^G$ is finite.
\end{proposition}
\begin{proof}
	The { main idea} of the proof { can be found} in  \cite[Theorem 4.1]{hv}. { However, the set of hypotheses in \cite{hv} is different from ours. A precise proof can be found in \cite[Theorem 6.5]{debanjanalim}. In fact, in the setting of Theorem 6.5 of \cite{debanjanalim}, we take $d=2$, $\alpha_1=m$, $F_n=F$ and $S$ as defined in \S\ref{section2} earlier. Then it follows from the first assertion of the above mentioned Theorem in \cite{debanjanalim} that the kernel and cokernel of the restriction map $R(E/F)\lra R(E/F_\infty)^G$ are finite.} 
\end{proof}
 	
 \begin{proposition}\label{prop2.2}
Assume that $R(E/F)$ is finite. If $R(E^k/F)$ is finite for every $k \geq 1$, then $\chi\big(G, R(E/F_\infty)\big)$ is finite.
\end{proposition}

\begin{proof}
     As $R(E/F)$ is finite, applying a control theorem ({ Proposition} \ref{thm:H_0}), we get that $R(E/F_\infty)^G$ is finite.  Further, it is well known that the finiteness of $R(E/F)$ implies that $R(E/F_\cyc)^\vee$ is $\Z_p[[\Gamma]]$ torsion and hence $H^2(F_\cyc, E_{p^\infty})=0$. Consequently, $H^2(F_\infty, E_{p^\infty})=0$ as well and by \cite[Lemma 3.1]{cs}, we deduce that $R(E/F_\infty)^\vee$ is a torsion $\Z_p[[G]]$  module. Also note that $G$ has $p$-cohomological dimension $=2$. Thus, to show { the} Euler characteristic is finite, it suffices to show $R(E/F_\infty)^G$ is finite and $H^2(G, R(E/F_\infty))$ is  finite (by Theorem \ref{theo:Ho}).

Next from the equation \eqref{eq5}, it is easy to see that $\chi(G, H^2_{\mathrm{Iw}}(T_pE/ F_\infty)^\vee)$ exists if and only if $\chi(G, R(E/F_\infty))$ exists. From the same equation, we can also deduce that $H^2_{\mathrm{Iw}}(T_pE/ F_\infty)$ is torsion over $\Z_p[[G]]$ (respectively $(H^2_{\mathrm{Iw}}(T_pE/ F_\infty)^\vee)^G$ is finite) if and only if $R(E/F_\infty)^\vee$   is torsion over $\Z_p[[G]]$ (respectively $R(E/F_\infty)^G$ is finite). 
From these discussions, we are reduced to show  $H^2(G, R(E/F_\infty))$ is finite or equivalently $H^2(G, H^2_{\mathrm{Iw}}(T_pE/ F_\infty)^{\vee})$ is finite.
As $G\cong \Z_p\rtimes\Z_p$, using Hochshild-Serre spectral sequence \cite[Page-119]{nsw} $H^2(G, H^2_{\mathrm{Iw}}(T_pE/ F_\infty)^\vee)\cong H^1(\Gamma, H^1(H, H^2_{\mathrm{Iw}}(T_pE/ F_\infty)^\vee)) $. Further, { $H^1(\Gamma, H^1(H, H^2_{\mathrm{Iw}}(T_pE/ F_\infty)^\vee)) $} is finite if and only if $H^0(\Gamma, H^1(H, H^2_{\mathrm{Iw}}(T_pE/ F_\infty)^\vee)) $ is  finite.

Thus we are further reduced to show $H^0(\Gamma, H^1(H, H^2_{\mathrm{Iw}}(T_pE/ F_\infty)^\vee)) $ is finite.  By \cite[Proposition 4.2]{ka}, we have a filtration of $H^1(H, H^2_{\mathrm{Iw}}(T_pE/ F_\infty)^\vee)$ given by  $0=S_0 \subset S_1\subset \cdots \subset S_K =H^1(H,H^2_{\mathrm{Iw}}(T_pE/ F_\infty)^\vee)$ such that $S_{i}/S_{i-1} \cong T(\chi^{s})$, where $T$ is a subquotient of $H^0(H, H^2_{\mathrm{Iw}}(T_pE/ F_\infty)^\vee) $ and $s \in \N$. 
Now by an argument similar to \cite[Page 562, last paragraph]{ka} which uses Nekovar's spectral sequence \cite[Proposition 8.4.8.3]{ne}, we get that $H^0(H,H^2_{\mathrm{Iw}}(T_pE/ F_\infty)^\vee)  \cong H^2_{\mathrm{Iw}}(T_pE/ F_\cyc)^\vee$. 

Then from the proof of \cite[Theorem 5.1]{ka}, it  suffices to show that for every $\Z_p[[\Gamma]]$ subquotient $T$ of $H^0(H, H^2_{\mathrm{Iw}}(T_pE/ F_\infty)^\vee) $ and for every $k \in \N$, $H^0(\Gamma, T(k)^\vee)$ is finite. Notice that, by our assumption, $H^2_{\mathrm{Iw}}(T_pE/ F_\cyc)$ is $\Z_p[[\Gamma]]$ torsion and hence  $T(k)$ is  $\Z_p[[\Gamma]]$ torsion. We fix  a topological generator $\gamma$ of $\Gamma$. 
Now from the structure theorem of finitely generated torsion modules over $\Z_p[[\Gamma]]$,  $H^0(\Gamma, T(k)^\vee)$ is finite if and only if $\gamma-1$ does not divide the characteristic ideal of $T(k) $. Note that  $T(k) $ is a subquotient of $H^2_{\mathrm{Iw}}(T_pE/ F_\cyc)(k)$, whence it is enough to prove that $\gamma-1$ does not divide characteristic ideal of $H^2_{\mathrm{Iw}}(T_pE/ F_\cyc)(k)$ (also see the proof of \cite[Theorem 5.1]{ka}) i.e.,\\$H^0(\Gamma, (H^2_{\mathrm{Iw}}(T_pE/ F_\cyc)(k))^\vee)$ is finite, for every $k\in \N$.

The finiteness of $H^0(\Gamma, (H^2_{\mathrm{Iw}}(T_pE/ F_\cyc)(k))^\vee)$ is  equivalent to the finiteness of $H^0(\Gamma, R(E/F_\cyc)(k))$, again by equation \eqref{eq5}.  As $\chi_p$ is trivial on $G_{F_\cyc}$, we have $H^0(\Gamma, R(E/F_\cyc))(k))\cong H^0(\Gamma, R(E^k/F_\cyc))$. 

{ Under the assumption $R(E^k/F)$ is finite for every $k\geq 1$, it reduces to show that the kernel and the cokernel of the restriction map $R(E^k/F)\to R(E^k/F_\cyc)^\Gamma$ are finite for every $k \geq 1$. For the remaining part of the proof, we establish this. For each  prime $v$ in the finite set $S$, choose a prime $w_c\mid v$ in $K_\cyc$ and denote by $\Gamma_v$ the decomposition subgroup of $\Gamma$ corresponding the primes $w_c\mid v$. By a standard  diagram chasing argument (see for example, \cite[Proposition 4.1]{wu}), it suffices to show that  $H^1(\Gamma, (E_{p^\infty}(k))^{G_{F_\cyc}})$ and $H^1(\Gamma_v, (E_{p^\infty}(k))^{G_{F_{\cyc,w_c}}})$, for each $v\in S$, are finite. Further,   $\Gamma$ being topologically cyclic, it is enough to show that $\big(E_{p^\infty}(k))^{G_{F}}$ and $ \big(E_{p^\infty}(k)\big)^{G_{F_{v}}}$, for each $v\in S$, are finite.

Consider a prime $v\mid p$ of $F$.  Since $E$ has  good reduction at $v$, we get  by Imai's theorem \cite[Theorem, page-1]{imai}  that $E_{p^\infty}(F_{\cyc,w_c})$ is finite. Observe that $\big(E_{p^\infty}(k)\big)^{G_{F_{\cyc,w_c}}}\cong \big(E_{p^\infty}(F_{\cyc,w_c})\big)(k)$. Hence, $ \big(E_{p^\infty}(k)\big)^{G_{F_{v}}}$ and also $ (E_{p^\infty}(k))^{G_{F}}$ are finite.

Next, we consider a prime $v\nmid p$ of $F$.  By moving to a finite extension of $F$, if necessary,  we can assume without any loss of generality that $E$ has either good or split multiplicative reduction at $v$.

Consider the case where $E$ has split multiplicative reduction at $v \nmid p$. Recall that $\mu_p \subset F$. It follows from the theory of Tate curves (see for example, \cite[Proposition 5.1(ii)]{hm}) that $E_{p^\infty}(F_{\cyc,w_c}) \cong \mu_{p^\infty}\oplus A $, where $A$ is a finite group. Twisting by $\chi_p^k$ with $k \geq 1$, we get that $E_{p^\infty}(F_{\cyc,w_c})(k) \cong \mu_{p^\infty}(k)\oplus A(k) $. Observe that the topologically cyclic group $\Gamma_v$ is  acting non-trivially on $ \mu_{p^\infty}(k) \cong \frac{\mathbb Q_p}{\mathbb Z_p}(k+1)$, which  is a cofree $\Z_p$ module of co-rank $1$. It follows that  $ \big(E_{p^\infty}(k)\big)^{G_{F_{v}}}$ is finite in this case. 

Now assume $E$ has good reduction at a prime $v \nmid p$ in $S$. Recall that  $V_pE:= (\underset{n}{\varprojlim}E_{p^n})\otimes_{\Z_p}\Q_p$. Note that $(E_{p^\infty}(k))^{G_{F_{v}}}$ being finite is equivalent to $\big(V_pE(k)\big)^{G_{F_{v}}}=0$. We have  $V_pE(k)$ is unramified at $v$ and by the Weil conjectures,  the complex absolute value of each of the eigenvalue of $\mathrm{Frob}_v$ acting on $V_pE(k)$ is $q_v^{k+1/2}$, where $q_v=\# O_F/v$. Hence  we deduce that $\big(V_pE(k)\big)^{G_{F_v}}=0$. This completes the proof.} 
\end{proof}
 Next{,} we discuss Jannsen's Conjecture \cite[Question 2, Page-349]{j}. Jannsen formulates this as  \enquote{Question 2} in his article but as we show in Remark \ref{remark1.9}, in our setting of Tate twists of an elliptic curve, \cite[Question 2, Page-349]{j} is equivalent to the Conjecture 1 stated by Jannsen \cite[page 317]{j}. Also note that Bellaiche \cite[Proposition 5.1]{be} states Question 2 of Jannsen as a conjecture. 
 
 Recall that certain finite set $S$ of primes of $F$ has been chosen at the beginning of this section. 
 For a Galois module $V$, set $V^\ast:=\Hom(V,\Q_p)$.
 \begin{conj}\label{ja} \cite[Question 2]{j}	Let $V$ be a $p$-adic representation coming from geometry of $G_{S}(F)$ which is pure of weight $w \neq - 1$ let $W =V^\ast(1).$ Then the natural map $H^1(G_{S}(F), W) \lra \underset{v\in S}\prod H^1(G_v, W)$ is injective.
\end{conj}
	\begin{proposition}\label{be}\cite[Proposition 5.1]{be}
	Let V be a p-adic geometric representation of $G_{S}(F)$ and let $W =V^\ast(1).$ The following are equivalent:
	\begin{enumerate}[(i) ]
		\item The natural map $H^1(G_{S}(F), W) \lra \underset{v\in S}\prod H^1(G_v, W)$ is injective.
		\item  \begin{small}{$\dim H^1(G_{S}(F),V)=[F:\Q] \dim V + \dim H^0(G_{S}(F), V) - \dim H^0(G_{S}(F), V^\ast(1))+\underset{ v\in S ,\ v \text{ finite} }\sum \dim H^0(G_v, V^\ast(1))- \underset{v\mid \infty}\sum \dim H^0(G_v,V)$.}\end{small}
		\item  \begin{small}{$\dim H^2(G_{S}(F),V)=\underset{ v\in S ,\ v \text{ finite} }\sum \dim H^0(G_v, V^\ast(1))-\dim H^0(G_{S}(F), V^\ast(1))$}\end{small}
		
	\end{enumerate}
	Moreover, in (ii) and (iii), the LHS is never less that the RHS. \qed
\end{proposition}
\begin{remark}\label{remark1.9}
Let $k$ be a positive integer.
By \cite [Theorem 5(a)] {j}, the statement (iii) in  Proposition \ref{be} for $V=V_pE^\ast(1-k)$   is equivalent to $H^2(G_{S}(F),V_pE^\ast(1-k))=0$. On the other hand, \cite[Conjecture 1, page 317]{j} states that  with  $k$ as above, $H^2(G_{S}(F),V_pE^\ast(1-k))=0$. Thus the statement in  Proposition  \ref{be}(i)  for $V_pE^\ast(1-k)$ with  $k \geq 1$ is equivalent to the \cite[Conjecture 1, page 317]{j}.
\end{remark}
Assuming  Conjecture {\ref{ja}}, the main result of \S \ref{section2} follows from Proposition \ref{prop2.2}:
 \begin{theorem}\label{mainnumber}
 Let $E/F$ be an elliptic curve with good reduction at all the primes of $F$ dividing $p$. Assume that $R(E/F)$ is finite and  Jannsen's Conjecture \ref{ja} is true. Then,  $\chi\big(G, R(E/F_\infty)\big)$ is finite. 
\end{theorem}
\begin{proof}
By Proposition \ref{prop2.2}, it suffices to show that $ R(E^k/F)$ is finite for every $k \geq 1$.
For each $k \geq 1$,  recall that the pure weight of $V_pE^\ast(1-k)$ is $2k-1$ (see \cite[Page 8]{be}).
As $2k-1\neq -1$, by  Conjecture \ref{ja}, we get that  \begin{small}$H^1(G_{S}(F), V_pE(k) ) \overset{\phi_k}\lra \underset{v\in S}\prod H^1(G_v,V_pE(k) )$\end{small} is injective. It follows that  \begin{small}$R(E^k/F)$ = $\ker \big(H^1(G_{S}(F), E_{p^\infty}(k)) ) \rightarrow \underset{v\in S}\prod H^1(G_v, E_{p^\infty}(k) \big)$\end{small} is finite. Indeed, ker$(\phi_k)$ surjects onto the divisible part of $ R(E^k/F)$. This completes the proof.
\end{proof}

\begin{remark}\label{rem1.11} Jannsen has shown \cite[Theorem 4]{j} that the analogue of his Conjecture \ref{ja} holds true for   function field{s}  of characteristic  $\ell\neq p$. Using \cite[Theorem 4]{j}, we prove the existence of   $\chi(G,R(E/F_\infty))$ for the false  Tate extension in  function field{s} of characteristic  $\ell\neq p$ setting, in Theorem \ref{theo:Janfunction}.

\end{remark}

 We continue to discuss  $\chi(G,R(E/F_\infty))$ under different sets of hypotheses. 

\begin{lemma} \label{lem2.7}
Assume that $R(E/F)$ is finite. Then $\chi(G,R(E/F_{\infty}))$ exists if and only if $H^1\big(H,R(E/F_{\infty})\big)^{\Gamma}$ is finite.
\end{lemma}
\begin{proof}
      Note that $G$ has $p$-cohomological dimension 2 and by Proposition \ref{thm:H_0} $H^0(G,R(E/F_\infty))$ is finite. 
{ Further as noted in the proof of Proposition \ref{prop2.2}, $R(E/F)$  is finite implies that $R(E/F_\cyc)^\vee$ and $R(E/F_\infty)^\vee$ are finitely generated torsion modules over $\Z_p[[\Gamma]]$  and $\Z_p[[G]]$, respectively.} Hence to show $\chi(G,R(E/F_{\infty}))$ is finite it suffices to show that $H^1(G,R(E/F_\infty))$ is finite  (by Theorem \ref{theo:Ho}). Now from the Hochschild-Serre spectral sequence, we have
      \begin{small}
      \begin{equation}\label{equation11}
      0\rightarrow H^1(\Gamma, R(E/F_{\infty})^H)\rightarrow H^1(G,R(E/F_{\infty})) \rightarrow H^1(H,R(E/F_{\infty}))^{\Gamma} \rightarrow 0
      \end{equation}
  \end{small}

      \noindent { 
      Notice that $F_\infty/F$ satisfies the conditions in \cite[Lemma 3.2]{cs} and from the proof of that lemma, we obtain that the kernel and the cokernel of the map $R(E/F_\cyc)\lra R(E/F_\infty)^H$ are cofinitely generated $\Z_p$ modules. It follows that $R(E/F_\infty)^H$ is a cotorsion $\Z_p[[\Gamma]]$ module. Then by applying Lemma \ref{theo:Ho}, we obtain that the finiteness of $H^1(\Gamma,R(E/F_{\infty})^H)$ is equivalent to the finiteness of $H^0(\Gamma, R(E/F_\infty)^H)=H^0(G,R(E/F_\infty))$. 
    Consequently,} from \eqref{equation11}, we get that $H^1(G,R(E/F_{\infty}))$ is finite if and only if $H^1(H,R(E/F_{\infty}))^{\Gamma}$ is finite.
\end{proof}
\begin{proposition}\label{lemma2.8}
		Let $R(E/F)$ be  finite. Also assume that $\corank_{\Z_p}(R(E/\f)^H)=\corank_{\Z_p}(R(E/\f)_H)=r$ (say) with $r\leq 1$. Then $\chi(G,R(E/F_{\infty}))$ exists and $\chi(G,R(E/F_{\infty}))=1 $.
\end{proposition}
\begin{proof}
	As $R(E/F_\infty)^H$ is a co-finitely generated $\Z_p$ module, it follows  by Nakayama lemma that $R(E/F_\infty)$ is a co-finitely generated $\Z_p[[H]]$ module.
	Consider the exact sequence of $\Z_p[[\Gamma]]$ modules:
   \begin{small}
	\begin{equation} \label{eq15}
		0\lra \ker(g)\lra (R(E/\f)^{\vee})^H \overset{g}\lra (R(E/\f)^{\vee})_H \lra \text{coker}(g) \lra 0, 
	\end{equation}
\end{small}
where $g$ is the composition of the natural maps  $(R(E/\f)^{\vee})^H\inj R(E/\f)^{\vee}\surj (R(E/\f)^{\vee})_H $.
 { Recall that $\Z_p[[\Gamma]]$ has Krull dimension $2$, so a $\Z_p[[\Gamma]]$ module is pseudonull if and only if it is finite.}  
We claim that the map $g$ is a pseudo-isomorphism of $\Z_p[[\Gamma]]$ modules {, i.e.,  $\ker(g)$ and $\coker(g)$ are pseudonull $\Z_p[[\Gamma]]$ modules or equivalently, they are finite}. Assume the claim for a moment. Then from  \eqref{eq15}, we deduce that the characteristic ideal of $(R(E/\f)^{\vee})^H $ and $(R(E/\f)^{\vee})_H $ as $\Z_p[[\Gamma]]$ modules are the same. By Proposition \ref{thm:H_0}, $\big((R(E/\f)^{\vee})_H\big)_{\Gamma} $ is finite and hence  $\big((R(E/\f)^{\vee})^H\big)_{\Gamma}= \big(H^1(H,R(E/\f))^{\Gamma}\big)^{\vee}$ is finite as well. Thus  $\chi(G,R(E/\f))$ exists by Lemma \ref{lem2.7}. Moreover, by our claim the Akashi series (\cite[Page-177]{cfksv}) of $R(E/F_\infty)$,     $Ak_H^G(R(E/F_\infty))=\frac{\text{char}_{\Z_p[[\Gamma]]} (H_0(H,R(E/\f)^{\vee})}{ \text{char}_{\Z_p[[\Gamma]]} (H_1(H,R(E/\f)^{\vee})}=1 $. Consequently, $\chi(G,R(E/\f))=1$ follows from  \cite[Lemma 3.6]{js}.

\medskip

For the rest of the proof, we establish the claim:  the kernel and the cokernel of  $g$ in \eqref{eq15} are finite.
 We only need to consider the case when $\corank_{\Z_p}(R(E/\f)^H)=\corank_{\Z_p}(R(E/\f)_H)=1$.
 Further, it is enough to show $\ker(g)$ is finite in order to deduce that $\text{coker}(g)$ is finite. Identifying $\Z_p[[H]] \cong \Z_p[[X]]$, by the structure theorem of $\Z_p[[X]]$ modules, there is a $\Z_p[[X]]$ module homomorphism 
 \begin{small}{
	\begin{equation}\label{equation13}
		R(E/\f)^{\vee} \lra   \bigoplus\limits_{i=1}^{s} \frac{\Z_p[[X]]}{P_i^{n_i}}
	\end{equation}
}\end{small}
\noindent with finite kernel and cokernel, where $P_i$'s are height one primes in $\Z_p[[X]]$.  { For a $\Z_p[[X]]$ module $M$, set $\text{ann}_X (M):=\{m\in M\ |\ X.m=0\}$,  the elements of $M$ that are annihilated by $X$.} Notice that  $\ran_{\Z_p} \big(\bigoplus\limits_{i=1}^{s}\frac{\Z_p[[X]]}{(P_i^{n_i},X)}\big)=\ran_{\Z_p} (\bigoplus\limits_{i=1}^{s}{ \text{ann}_X\big(\frac{\Z_p[[X]]}{P_i^{n_i}}\big)})=1$. Therefore, there exists $1\leq k\leq s$ such that $P_k=(X)$, $n_k=1$ and $P_i\neq (X)$, for $i\neq k$. Consider the commutative diagram: 
   \begin{small}
	\begin{equation}\label{equation14}
		\begin{tikzcd}
			0\arrow{r}
			&\ker(f_1)\arrow{r}
			&(R(E/\f)^{\vee})_{H}\arrow{r}{f_1}
			& \bigoplus\limits_{i=1}^{s}\frac{\Z_p[[X]]}{(P_i^{n_i},X)}
			\\
			0\arrow{r}
			&\ker(f_2)\arrow{r}\arrow{u}{g_1}
			&(R(E/\f)^{\vee})^{H} \arrow{r}{f_2}\arrow{u}{g}
			&\bigoplus\limits_{i=1}^{s}{ \text{ann}_X\big(\frac{\Z_p[[X]]}{P_i^{n_i}}\big)} \arrow{u}{\bigoplus\limits_{i=1}^{s} g_{2,i}}
		\end{tikzcd}
	\end{equation}
\end{small}
	For an element $x\in { \text{ann}_X\big(\frac{\Z_p[[X]]}{P_i^{n_i}}\big)}$, $g_{2,i}(x)$  denotes the residue class of $x$ in $\frac{\Z_p[[X]]}{(P_i^{n_i},X)}$. The map $g_1$ is induced by the restriction of $g$. The map $g_{2,k}$ is an isomorphism and for $i\neq k$, ${ \text{ann}_X\big(\frac{\Z_p[[X]]}{P_i^{n_i}}\big)}=0$. Thus, $\ker(\bigoplus\limits_{i=1}^{s} g_{2,i})=0$. Also, from \eqref{equation13}, we know that  $\ker(f_2)$ is finite. Consequently, from  \eqref{equation14}, we deduce that $\ker(g)$ is finite. This establishes the claim.   
	\end{proof}

	From now onwards, we assume that $E/F$ has (good) {\it ordinary reduction at all the primes of $F$ dividing $p$}.
We recall the following set of primes of $F_\cyc$ associated to  $E$ from \cite{hv}.  

\begin{small}

$P_0:=\{u \text{ is a prime in } F_\cyc : \ u\nmid p\text{ and } u\mid m  \} ,$

$P_1:=\{u\in P_0: E/F_\cyc \text{ has split multiplicative reduction at } u  \} $ and  $m_1:=\# P_1$,

$P_2:=\{u\in P_0: E \text{ has good reduction at } u \text{ and } E(F_{\cyc,u})_{p^\infty}\neq 0  \}$ and   $m_2:=\#P_2$.
\end{small}

\begin{proposition} \label{thm:euler}
	Assume that $R(E/F)$ is finite and  $S(E/F_{\infty})^\vee$ is a finitely generated  $\mathbb{Z}_p[[H]]$ module of rank $1$. Then $\chi(G,R(E/F_{\infty}))$ exists.
\end{proposition}
\begin{proof}
	Let $f$ be the natural inclusion map  $R(E/F_\infty) \overset{f}\hookrightarrow S(E/F_\infty)$. It is shown in \cite[proof of Theorem 5.3]{hv} that $H^1(H,S(E/\f))=0$. Hence $H^1(H,\text{coker}(f))=0$ and we also have:
	\begin{small}
	\begin{equation} \label{eq11}
	0\rightarrow R(E/F_{\infty})^{H}\overset{f_H}\longrightarrow S(E/F_{\infty})^{H}\rightarrow (\text{coker} (f))^{H}\rightarrow H^{1}(H,R(E/F_{\infty})) \rightarrow 0
	\end{equation}
\end{small}
	Now by \cite[Theorem 3.1]{hv}, $S(E/F_\infty)^{\vee}$ is a finitely generated $\Z_p[[H]] $ module and  there is an injective homomorphism  $S(E/\f)^\vee\inj \Z_p[[H]]^{\lambda + m_1 + 2m_2}$ with finite cokernel, where $\lambda=\corank_{\Z_p}(S(E/F_\cyc))$. By our hypothesis that $\corank_{\Z_p[[H]]}(S(E/F_\infty))=1$, we get that $m_2=0$  and $\lambda\in\{0,\ 1\}$. 
	
	First, we consider the case $\lambda=0$. Notice that, $R(E/\f)$ being a subgroup of $S(E/\f)$, $\corank_{\Z_p[[H]]}(R(E/\f)) \leq1$. 
	
In the subcase, when $\corank_{\Z_p[[H]]}(R(E/\f))=0$,
	 by Theorem \ref{theo:Ho}, we get that $\corank_{\Z_p}(R(E/\f)^H)=\corank_{\Z_p}(R(E/\f)_H)$. Again applying Theorem \ref{theo:Ho}, we obtain that $\corank_{\Z_p}(S(E/F_\infty)^H)$=1. Therefore{,} using \eqref{eq11}, we observe that all the hypotheses of Proposition \ref{lemma2.8} are satisfied and the result follows from the same proposition, in this subcase.

	Next, consider the subcase where $\lambda=0$ and $\corank_{\Z_p[[H]]}(R(E/\f))= 1$. Now by our hypothesis, $\corank_{\Z_p[[H]]}(S(E/\f))= 1$. Hence $\corank_{\Z_p[[H]]}(\text{coker}(f))$  $=0 $ and by Theorem \ref{theo:pseudo}, $\text{coker}(f)$ is a $\Z_p[[G]]$ pseudonull module. Further, by \cite[Remark 3.2]{hv}, the maximal pseudonull  submodule of $S(E/\f)^{\vee}$ is 0. Thus 
	 $\text{coker}(f)=0$ and $S(E/F_\infty)=R(E/F_\infty)$. As $\lambda=\corank_{\Z_p}(S(E/F_\cyc))=0$, via a control theorem for $F_\cyc/F$, we deduce that $S(E/F)$ is finite. Now applying \cite[Theorem 4.1]{hv},  $\chi(G, S(E/\f))=\chi(G,R(E/\f))$  is finite.

\medskip

Now, we consider the second case, where $\lambda=1$.
In this case, $m_1=m_2=0$ and $E$ has good reduction at primes $v\mid p$ of $F$. Then, by a control theorem similar to \cite[Lemma 2]{jh}, we get that $\cora_{\Z_p}(R(E/F_\cyc))=\cora_{\Z_p} (R(E/F_\infty)^H)$. Also, given $\lambda=\corank_{\Z_p}(S(E/F_\cyc))=1$, $\corank_{\Z_p}(R(E/F_\cyc))$ is atmost 1.
If $\corank_{\Z_p}(R(E/F_\cyc))=0$, then (by Theorem \ref{theo:Ho})    $\cora_{\Z_p}(R(E/\f)^H)=\cora_{\Z_p}(R(E/\f)_H)=0$. Once again, by applying Proposition \ref{lemma2.8}, we deduce that $\chi(G,R(E/\f))$ exists. On the other hand, if $\corank_{\Z_p}(R(E/F_\cyc))=\corank_{\Z_p}(R(E/F_\infty)^H)=1$, then by Theorem \ref{theo:Ho} $\corank_{\Z_p}(R(E/F_\infty)_H)$ could be $0$  or $\corank_{\Z_p}(R(E/F_\infty)_H)$ could be $ 1$. In the former case, $\chi(G,R(E/\f))$ exists by Lemma \ref{lem2.7}  and in the later case, the existence of $\chi(G,R(E/\f))$ follows from Proposition \ref{lemma2.8}. This completes the proof of this theorem.
	\end{proof}
We give an example  where all the hypotheses of  Proposition \ref{thm:euler} are satisfied:
\begin{example} 
	Let $p=3$, $F=\Q(\mu_3)$ and $F_\infty=\Q(\mu_{3^\infty}, 3^{1/3^{\infty}})$. Consider the following elliptic curve $E$ with LMFDB label 306.a2 :
	\begin{small}
	\begin{equation}
		y^2+xy=x^3-x^2-927x+11097
	\end{equation}
\end{small}
\noindent By discussions on \cite[Pages-129, 130]{gr3}, it follows that $rank(E(F))\geq 1$, $E$ has good, ordinary reduction at the prime of $K$ dividing $p$, $\mu$-invariant of $S(E/F_\cyc)^{\vee}$ vanishes and $\corank_{\Z_p}(S(E/F_\cyc))=1$. By \cite[Theorem 3.1]{hv}, we deduce that $\corank_{\Z_p[[H]]}(S(E/F_\infty))=1$. So, all the conditions of Proposition \ref{thm:euler} are satisfied. 
\end{example}
\begin{remark}
	Among the various cases discussed in the proof of { Proposition} \ref{thm:euler}, we have considered the situations where $\cora_{\Z_p[[H]]}(R(E/\f))=1$.
 Assume that Conjecture \ref{conjb} holds. Then these situations cannot occur. 
\end{remark}

\begin{corollary}\label{corollary2.13}
	Let us keep the hypotheses of Proposition \ref{thm:euler}. Further, assume that Conjecture \ref{conjb} holds. Then, by Proposition \ref{lemma2.8}, $\chi(G,R(E/\f))=1$.
\end{corollary}
\begin{remark}

 Let $G$ be a commutative 
compact $p$-adic Lie group without  any element of order
 $p$. Then it is well known that for a finitely generated pseudonull $\Z_p[[G]]$ module $M$  with well-defined Euler characteristic, $\chi(G,M)=1$ holds. In general, there are examples \cite[Example 3]{css} of non-commutative compact $p$-adic Lie group $G$ and finitely generated pseudonull $\Z_p[[G]]$ modules $M$ such that $\chi(G,M)\neq 1$. However, in the special case where $G\cong\Z_p\rtimes \Z_p$ and $M=R(E/F_\infty)^\vee$, we see in  Corollary \ref{corollary2.13} that the pseudonullity of $R(E/F_\infty)^\vee$ implies $\chi(G, R(E/\f))=1$.
\end{remark}
The following variant of Proposition \ref{lemma2.8} can be proved easily.
\begin{lemma}
	Let   $m_1=m_2=0$ and assume that 
$R(E/F_\cyc)$ is finite. Then $R(E/F_\infty)$ is a pseudonull $\Z_p[[G]]$ module and $\chi(G,R(E/F_\infty))$ exists. \qed
	\end{lemma}

\noindent Next, we discuss $\chi(G, R(E/F_\infty))$ when $\cora_{\mathbb{Z}_p[[H]]} S(E/F_{\infty})=2$:

\begin{proposition}\label{thm2.10}
	Suppose  $\ran_{\Z}(E(F))> 0$. Also assume  that $R(E/F)$ is finite, $\mu$-invariant of $S(E/F_\cyc)^{\vee}$ vanishes, $\cora_{\mathbb{Z}_p[[H]]} S(E/F_{\infty})=2$ and $m_1=m_2=0$. Then $\chi(G,R(E/F_{\infty}))$ exists. 
\end{proposition}
\begin{proof}
	Since $m_1=m_2=0$, we get from \cite[Theorem 3.1]{hv} that $\cora_{\Z_p}(S(E/F_\cyc))$  $=2$. Then by applying \cite[Corollary 4.4 and Proposition 4.9 ]{cs}, we further deduce that $\cora_{\Z_p}(R(E/F_\cyc))\leq 1$ and $\cora_{\Z_p[[H]]}(R(E/\f))\leq 1$.
	
	Now, the result can be deduced following the proof of Proposition \ref{thm:euler}.
\end{proof}

We also briefly mention the commutative case with $G \cong \Z_p^{d+1}$.

\begin{proposition}\label{prop2.17}
	Let $F$ be a totally real field of degree $d$ over $\QQ$ and $K$ be  a CM field which is a quadratic extension of $F$. Let $K_\infty$ be a Galois extension of $K$, such that $G:=\gal(K_\infty/K)\cong \Z_p^{d+1}$.  
	Let $E$ be an elliptic curve over $K$.  Then $\chi(G,R(E/K_\infty))$ exists if $R(E/K)$ is finite.
\end{proposition}
\begin{proof}
As $G$ is commutative, it suffices to show $R(E/K_\infty)^G$ is finite. Finiteness of $R(E/K_\infty)^G$ follows from  \cite[Example 1.9 and the proof of Theorem 1]{jo}.
\end{proof}

\begin{remark}\label{hmfalse}
	The Euler characteristic of the fine Selmer group of elliptic curve over the false Tate-curve extension has been discussed in \cite{ha}. Assume that $R(E/F)$ is finite, then it follows  from \cite[Theorem 3.1]{ha}, that $\chi(G,R(E/\f))$ exists if and only if the kernel of the natural
	map $\phi_{F_\infty/F} $ (see \cite [equation 71]{cs}),
	\begin{small}
	\begin{equation}\label{equation2.1}
		H^1_{\mathrm{Iw}}(T_pE/ F_\infty)_G \overset{\phi_{F_\infty/F}} \lra H^1_{\mathrm{Iw}}(T_pE/F)=H^1(F_S/F,T_pE)
	\end{equation}
\end{small}
	is finite.  
	
	However, applying Nekovar's spectral sequence we actually get that \cite[Proposition 8.4.8.3]{ne} $\text{ker}\Big(H^1_{\mathrm{Iw}}(T_pE/ F_\infty)_G \lra H^1(F_S/F,T_pE)\Big)=H_2(G,H^2_{\mathrm{Iw}}(T_pE/ F_\infty))$. By Poitou-Tate exact sequence \eqref{eq5}, we have seen that the finiteness of\\ $H_2(G,H^2_{\mathrm{Iw}}(T_pE/ F_\infty))$ is equivalent to the finiteness of $H^2(G, R(E/F_\infty))$. 
	
	Thus, following the criterion of \cite[Theorem 3.1]{ha}, establishing the finiteness of  $H_2(G,H^2_{\mathrm{Iw}}(T_pE/ F_\infty))$ seems as difficult as showing the existence of $\chi(G, R(E/F_\infty))$.

\end{remark}

We now discuss a criterion for the existence of $\chi(G, R(E/F_\infty))$ in terms of the map  $\phi_{F_\infty/F_\cyc} : H^1_{\mathrm{Iw}}(T_pE/ \f)_H\lra H^1_{\mathrm{Iw}}(T_pE/ F_\cyc)$.
\begin{proposition}\label{thm2.12}
Assume that  $m_1=m_2=0$, $R(E/F)$ is finite and $S(E/\f)^\vee $ is a finitely generated $\Z_p[[H]]$ module. Further assume  that $\ran_{\Z_p}(\text{coker}(\phi_{F_\infty/F_\cyc}))$ =0. Then $\chi(G,R(E/F_{\infty}))$ exists.
\end{proposition}
\begin{proof}
	By Lemma \ref{lem2.7}, it suffices to show that $H^1(H, R(E/\f))=0$. Using \eqref{eq11}, it further reduces to show that $\cora_{\Z_p} (\text{coker}(f)^H)$=$r-s$, where $r=\corank_{\Z_p}S(E/\f)^H$ and  $s$=$\corank_{\Z_p}R(E/\f)^H$. 	As $H^1(H,S(E/\f))$=$0$ \cite[proof of Theorem 5.3]{hv}, we observe that  $r$=$\cora_{\Z_p[[H]]}S(E/\f)$ and it follows from \eqref{eq11} that $\cora_{\Z_p}(\text{coker}(f)^H)$=$\cora_{\Z_p[[H]]}\text{coker}(f)$.

	As $m_1$=$m_2$=$0$, using a control theorem for  $R(E/F_\cyc) \rightarrow R(E/F_\infty)^H$, we have  $\cora_{\Z_p}R(E/\f)^H$=$\cora_{\Z_p}R(E/F_\cyc)$. Finally,  we use the hypothesis $\ran_{\Z_p}\text{coker}(\phi_{F_\infty/F_\cyc})$=$0$ and deduce  from \cite[Theorem 4.11]{cs} that $\cora_{\Z_p}R(E/F_\cyc)$ =$\cora_{\Z_p[[H]]}R(E/\f)$. Therefore, $\corank_{\Z_p}(\text{coker}(f)^H)$=$r-s$. 
\end{proof}

\begin{remark}
	We keep the hypotheses and setting of Proposition \ref{thm2.12}. 
	Then from \cite[Theorem 4.11]{cs}, it follows that if we assume Conjecture \ref{conjb} holds, then the following conditions are equivalent: 
	\begin{enumerate}[(i)]
		\item $\ran_{\Z_p}(\text{coker}(\phi_{F_\infty/F_\cyc}))=0$.
		\item $R(E/F_\cyc)$ is finite.
	\end{enumerate}
\end{remark}

\begin{remark}\label{limfalse} 
In \cite{li}, the  Euler characteristic of the fine Selmer group of an abelian variety is discussed 
and contains the following corollary. 

\medskip

\noindent \begin{small}{\cite[{\bf Corollary 3.5}]{li} {\it 	Let $A$ be an abelian variety over $K$. Let $T$ be the Tate module of the dual of the abelian variety $A^*$ of A. Let $\K$ be a $p$-adic Lie extension of a number field $K$ with Galois group $G$. Suppose that there is a finite family of closed normal subgroups $G_i(0\leq i \leq r)$ of $G$ such that $1=G_0\subseteq G_1\subseteq\cdots \subseteq G_r=G$,  $G_i/G_{i-1}$ $\cong$ $\mathbb{Z}_p$, for every $i$. Then if $Y(T/K)$ is finite, the $G$-Euler characteristic of $Y(T/\K)$ is defined. }\qed }\end{small}

\medskip

\noindent The following result of Kato was used crucially in proving \cite[Corollary 3.5]{li}.

\medskip

\noindent \begin{small}{\cite[{\bf Proposition 2.3}]{li} (\cite[Proposition 4.2]{ka})  
	{\it Let $G$ be a compact pro-$p$ $p$-adic Lie group without $p$-torsion. Let $N$ be a closed
	normal subgroup of $G$ such that $G/N$ has no $p$-torsion. Suppose that there is a finite family of closed normal subgroups $N_i$ $(0\leq i \leq r)$ of $G$ such that $1=N_0\subseteq N_1\subseteq\cdots\subseteq N_r=N$,  $N_i/N_{i-1}$ $\cong$ $\mathbb{Z}_p$ for $1\leq i \leq r$ and such that the action of $G$ on $N_i/N_{i-1}$ by inner automorphisms is given by a homomorphism $\chi_i:G/N\longrightarrow \mathbb{Z}_p^{\times} $. Let $M$ be a finitely generated $R[[G]]$ module. Then the following statements are equivalent.
	\begin{enumerate}[i]
		\item . $M_N$ is a torsion $R[[G/N]]$ module.
		\item . $M$ is a torsion $R[[G]]$ module, and $H_n(N,M)$ is a torsion $R[[G/N]]$ module for every $n\geq 1$. \qed
	\end{enumerate}
}}\end{small}
Taking $G=N$ in \cite[Proposition 2.3]{li}, the author obtains \cite[Corollary 3.5]{li} via \cite[Theorem 3.4]{li}. 

However, the  assumption $G=N$ forces  the action of $G$ on $N_i/N_{i-1}$ by inner automorphism $\chi_i$ to be trivial for all $i$. Note that, this in particular implies that $G$ is abelian. Hence, \cite[Proposition 2.3]{li} do not apply for $G=N$ in the case when $G$ is not abelian, in particular for the false Tate curve extension case with $G\cong\Z_p\rtimes \Z_p$. Thus, the proof of \cite[Theorem 3.4]{li} and hence \cite[Corollary 3.5]{li} seems to have some gap 
 for the  false Tate curve extension.
\end{remark} 
\section{Fine Selmer group over function fields of characteristic $\ell \neq p$}\label{section4}
We choose and fix a rational prime $\ell $ distinct from $p$ and   take $r \in \mathbb N$ such that the finite  field $\FF =\FF_{\ell^r}$ contains $\mu_p$. Set  $K=\FF (t)$. Let $\FF^{(p)}$ be the unique (unramified) $\Z_p$ extension of $\FF$ contained in $\overline{\FF}_{\ell}$. Note that $\FF^{(p)}(t)= K(\mu_{p^\infty})$ and we denote $\FF^{(p)}(t)$ by $K_\cyc$. Further, we choose any non-constant polynomial $q(t)$ in $K$ and put $K_\infty:=\underset{n\geq 0}\cup \FF^{(p)}(t)(q(t)^{\frac{1}{p^n}})$. Then $G:=\gal(K_\infty/K)\cong \Z_p\rtimes \Z_p$ \cite[Remark 3.5]{w}, $H:=\gal(K_\infty/K_\cyc)\cong \Z_p$ and $\Gamma=G/H\cong\gal(K_\cyc/K)\cong \Z_p$. $K_\infty$ is an analogue of the false-Tate extension of number fields. Throughout section \ref{section4}, $E$ will be a non-isotrivial elliptic curve defined over $K$, i.e, $j(E) \notin \FF$ and S will be a finite set of primes of $K$ containing the primes of bad reduction for $E$ and the primes which ramify over $K_\infty$. In this section, we discuss $p^\infty$-fine Selmer groups of $E$ over $K_\infty$. 

For a finite extension $L/K$ and each $v \in S$, put $K^1_v(E/L):= \underset{w\mid v}\prod H^1(L_w,E_{p^\infty})$.  We define the $p^\infty$-fine Selmer group $R(E/L) $ of $E$  over $L$ as:
\begin{small}{
\begin{equation}\label{def3.1}
	R(E/L) :=\ker(H^1(G_S(L),E_{p^\infty})\lra \underset{v\in S}\bigoplus K^1_v(E/L)).
\end{equation}}
\end{small}
As before, the definition of $R(E/\mathcal L)$ extends to  an infinite extension $\mathcal L/K$ by taking { the} inductive limit over finite subextensions. 
\begin{remark}\label{rem5.1}
	Recall that the classical $p^\infty$-Selmer group of $E/L$ is defined as:
	\begin{small}
	\[
	S(E/L):=\Ker(H^1(G_S(L),E_{p^{\infty}})\longrightarrow \underset{w\mid S}\oplus H^1(L_w,E_{p^\infty})/\text{im}(\kappa_w) ,		  
	\]
	\end{small}
	where $\kappa_w:E(L_w)\otimes \Q_p/\Z_p\inj H^1(L_w, E_{p^\infty}) $ is the local Kummer map. Moreover, in this function field setting  with ${\ell}\neq p$, $\text{im}(\kappa_w)=0$ \cite[Prop.  3.3]{bl1}.
	
	Thus, we see that the classical Selmer group $S(E/L)$ is the smallest possible Selmer group and in fact we have  $S(E/L)=R(E/L)$.
	
	Note that the definition of $R(E/L)=S(E/L)$ is independent of the choice of $S$ as long as $S$ contains all the primes of bad reduction of $E/L$ \cite[Page 119]{va}.
\end{remark}	
\begin{remark}
We will discuss the Euler characteristic  $\chi(G,R(E/K_\infty))$ over the  extension $K_\infty/K$. In fact, as explained in Remark  \ref{rem1.11}, we will prove the existence of $\chi(G,R(E/K_\infty))$ using \cite[Theorem 4]{j}. 
Note that, for the cyclotomic case, the existence of $\chi(\Gamma, S(E/K_\cyc))$ is discussed under a variety of hypotheses in \cite{pa} (see \cite[Theorems 3.7 and  3.9]{pa}).
\end{remark}

 Let $\chi_p$
be the $p$-adic cyclotomic character.  
Given a $G_{K}$ module $M$,  we denote by $M(k)$ the twist of $M$ by $\chi_p^k$. 
Define $R(E^k/K)$ by replacing $E_{p^\infty}$ by  
$E_{p^{\infty}}(k)$ in the equation \eqref{def3.1}.
We begin with  the following lemma.
\begin{lemma}\label{lemma3.4}
	Let $K_\infty/K$ be the extension as defined above in \S \ref{section4}. Let $E/K$ be a non-isotrivial elliptic curve with  good or split multiplicative reductions at all primes of $K$.  If $R(E/K)$ is finite, then $H^0(G, R(E/K_\infty))$ is finite.
\end{lemma}
\begin{proof}
	Note that it is enough to establish  a control theorem by showing that the kernel and the  cokernel of the natural map $R(E/K)\lra R(E/K_{\infty})^{G}$ are finite. By the snake lemma, it suffices to show that $H^i(G, E_{p^\infty}(K_\infty)) $ is finite  for $i=1,2$ and $H^1(G_v, E_{p^\infty}(K_{\infty,w}))$ is finite for every place $w$ of $K$ with $w\mid v$, $v\in S$.
	The finiteness of these local and global terms  follows from \cite[Corollary 4.9, Lemmas 3.3 and 4.2]{bl2}.
\end{proof}
\begin{proposition}\label{proposition3.5}
	Let us keep the setting and hypotheses of Lemma \ref{lemma3.4}. If $R(E^k/K)$ is finite for every $k \geq 1$, then $\chi\big(G, R(E/K_\infty)\big)$ is finite. 
	
\end{proposition}
\begin{proof}
	By  Lemma \ref{lemma3.4}, we deduce that $H^0(G, R(E/K_\infty))$ is finite. Hence, to show that   $\chi(G, R(E/K_\infty))$ exists, it suffices to show that  $H^2(G, R(E/K_\infty))$	is finite (Theorem \ref{theo:Ho}). Note that the equation \eqref{eq5} continue to hold in this function field setting ($\ell \neq p$) for the extension $K_\infty$. Then using  \cite[Lemmas 4.2 and 3.3]{bl2} in \eqref{eq5}, we obtain that $H^2(G,R(E/K_\infty))$ is finite if and only if  $H^2(G, H^2_{\mathrm{Iw}}(T_pE/ K_\infty)^\vee) $ is finite. Now, using the techniques of Kato \cite[Theorem 5.1]{ka} as in the proof of Proposition \ref{prop2.2}, it is enough to show that $H^0(\Gamma, R(E/K_\cyc)^\vee(k))\cong H^0(\Gamma,R(E^k/K_\cyc)) $ is finite for every $k\geq 1$.  By a control theorem similar to \cite[Theorem 3.9]{pa},  the finiteness of $H^0(\Gamma,R(E^k/K_\cyc))$ is equivalent to the finiteness of $R(E^k/K)$. 
\end{proof}
\begin{theorem}\label{theo:Janfunction}
	Let $K_\infty/K$ be the false Tate curve extension as described above. Let $E/K$ be a non-isotrivial elliptic curve with  good or split multiplicative reductions at all primes of $K$.  If $R(E/K)$ is finite, then $\chi\big(G, R(E/K_\infty)\big)$ is finite.
\end{theorem}
\begin{proof}
	By Proposition \ref{proposition3.5}, it suffices to show that $R(E^k/K)$ is finite for every $k \geq 1$. For each $k$, set $V_k=V_pE(k)$. Then the pure weight of $V_k$ { is equal to}  $-2k-1$ (see \cite[Page-8]{be} and \cite[Page-356]{j}). From \cite[Theorem 4]{j},  for each $k\geq 1$, the map \begin{small}{$\phi_k:H^1(G_S(K),V_k)\lra \underset{v\in S}\prod H^1(K_v, V_k)$}\end{small} is injective.  As the kernel of $\phi_k$ surjects onto the divisible part of $R(E^k/K)$, it follows that $R(E^k/K)$ is finite.
\end{proof}

\begin{remark} {\bfseries{Comparison of Theorem \ref{theo:Janfunction} with the work of \cite{va}.}}
	The existence of $\chi(G, R(E/\mathcal L))$ for some compact p-adic Lie extension $\mathcal L$ of $K$ has been proved in \cite{va}, under {a} certain set of hypotheses. In particular, for the false Tate curve extension
	$K_\infty/K$, it is shown  in \cite[Theorem 1.2]{va}  that $ { \chi(G, R(E/K_\infty))} (=\chi(G, S(E/K_\infty)))$ exists under  the following assumptions: 
	
	\begin{enumerate}[(i)]
		\item $E$ has either good or split multiplicative reduction at every prime of $K$,
		\item $R(E/K)$ is finite,
		\item $H^2(G_S(K_\infty), E_{p^\infty})=0$,
		\item $\chi( G, E_{p^\infty}(K_\infty ))$ exists,
		\item for every $v\mid S$, $\chi( G_v, E_{p^\infty}(K_{\infty,v} ))$ exists, where $ G_v$ is the decomposition subgroup of $G$ at $v$, and
		\item the map $H^1(G_S(K_\infty), E_{p^\infty})\lra \underset{v\mid S}\oplus H^1( K_{\infty, v}, E_{p^\infty}) $ is surjective.
	\end{enumerate}
	
	In Theorem \ref{theo:Janfunction}, under the assumptions (i), (ii) and $j(E)\notin \FF$, we prove the existence of $\chi(G, R(E/K_\infty)) $.
	Our proof is different from the proof of \cite[Theorem 1.2]{va} and uses the results of Kato \cite{ka} and Jannsen \cite{j}. Moreover, we do not  need the assumptions (iii)-(vi) in our proof. 
	
\end{remark}

\subsection{Analogue of Conjecture A}
 In this function field setting ($\ell \neq p$), the  analogue of Conjecture \ref{conja}, i.e. that $R(E/K_\cyc)^\vee$ is a finitely generated $\Z_p$ module,  is known and we briefly explain it here. Note that for an elliptic curve $E/K$,  $R(E/K_\cyc)^\vee$ is a finitely generated torsion $\Z_p[[\Gamma]]$ module \cite[Theorem 4.24]{w}.

Let $K$ be  a function field of transcendence degree 1 over a finite field of characteristic ${\ell}$ and  let $U$ be an open dense subset of $C$, the proper smooth curve with function field $K$. For any algebraic extension $L/K$, Witte \cite{w} defined the $U$-Selmer group, $Sel_{U}(L,E_{p^\infty}) $  as follows: 
\begin{small}
\begin{equation}
	Sel_{U}(L,E_{p^\infty}):=\ker(H^1(L,E_{p^\infty})\lra \underset{v\in U^0_L}\oplus  K^1_v(E/L)), 
\end{equation}
\end{small}
where $v\in U^0_L $ varies over the closed points of normalisation of $U$ in $L$. Further  the Iwasawa $\mu$-invariant of $Sel_{U}(K_\cyc,E_{p^\infty})$  vanishes \cite[Corollary 4.38]{w}. Since  $R(E/K_\cyc)\subset Sel_{U}(K_\cyc,E_{p^\infty})$,  Conjecture  A is true for $R(E/K_\cyc)^\vee$.

\begin{remark}\label{mu0lneqp}
	Alternatively, {the fact that}  $R(E/K_\cyc)^\vee$ is a finitely generated $\Z_p$ module can be shown using the relation between the ideal class group and $R(E/K_\cyc)$. Let $E/K$ be a non-isotrivial elliptic curve with good or split multiplicative reduction at all places of $K$.  Let $L=K(E[p])$ and let $L\subset L_n\subset L_\cyc$ be such that $[L_n:L]=p^n$.  Let $Cl(F)$ denote the divisor class group of $F$, where $F$ is a finite extension of $K$. Then, using a proof similar to \cite[Theorem 5.5]{lm} and \cite[Proposition 11.16]{ro}, we see that  $R(E/L_\cyc)^\vee$ is a finitely generated $\Z_p$ module and consequently the  same holds for $R(E/K_\cyc)^\vee$ as well. 
\end{remark}
\subsection{Analogue of Conjecture B}
\noindent
{{ We claim}} that the kernel and the cokernel of the map $R(E/K^{\cyc})\rightarrow R(E/K_{\infty})^{H}$ are cofinitely generated $\Z_p$ modules.

{ Observe that we have defined the fine Selmer groups $R(E/-)$  by using Galois cohomology  in this setting of function fields of characteristic $\ell\neq p$. Therefore, we can use arguments similar to the proof of \cite[Lemma 3.2]{cs} to deduce this claim.
More precisely, consider the following commutative diagram:
\begin{equation} \label{diag:controlcyctofalsetate}
		\begin{tikzcd}
			0\arrow{r}
			&R(E/K_{\infty})^{H}\arrow{r}
			&H^1(G_S(K_{\infty}), E_{p^\infty})^H\arrow{r}
			&(\underset{v\in S}\oplus K^1_v(E/K_{\infty}))^{H}
			\\
			0\arrow{r}
			&R(E/K_\cyc)\arrow{r}\arrow{u}{\alpha}
			&H^1(G_S(K_\cyc), E_{p^\infty}) \arrow{r}\arrow{u}{\beta}
			&\underset{v\in S}\oplus K^1_v(E/K_\cyc).\arrow{u}{\gamma=\underset{v\in S}\oplus \gamma_v}
		\end{tikzcd}
	\end{equation}

As $H\cong \Z_p$ has $p$-cohomological dimension  $1$, $\coker{\beta}=0$. 
Applying a snake lemma to the diagram \ref{diag:controlcyctofalsetate}, it suffices to show  that $\ker(\beta)$ and $\ker(\gamma)$ are cofinitely generated $\Z_p$ modules to establish the claim.
Note that $ E_{p^\infty}(K_\infty)$ is a cofinitely generated $\Z_p$-module and we have $H\cong \Z_p$. Hence, 
 $\ker(\beta)\cong H^1(H, E_{p^\infty}(K_\infty))$ is a cofinitely generated $\Z_p$-module.  
For each prime $v\in S$, fix a prime $w$ in $K_\infty$ and choose $w_c$ in $K_\cyc$ such that $w\mid w_c\mid v$. Let $H_{w_c}$ be the decomposition subgroup of $H$ corresponding to the primes $w\mid w_c$.
Then by Shapiro's lemma (see \cite[Equation 81, Page-835]{cs}), we get that $\ker(\gamma)=\underset{v\in S}\oplus \ker(\gamma_v)$ with
$\ker(\gamma_v)\cong \underset{w_c\mid v}\oplus H^1(H_{w_c}, E_{p^\infty}(K_{\infty,w}))$.
Again, each of $H^1(H_{w_c}, E_{p^\infty}(K_{\infty,w}))$ is a cofinitely generated $\Z_p$ module.
Since there are finitely many primes in $K_\cyc$ lying over a prime $v\in S$ and $S$ is finite, we deduce that $\ker(\gamma)$ is a cofinitely generated $\Z_p$ module.
}

{ As}  $R(E/K_\cyc)$ is a finitely generated $\Z_p$ module, we arrive at the following:

\begin{proposition}\label{prop5.3}
	Let $K_\infty/K$ be as  above with $\gal(K_\infty/K)  \cong \Z_p\rtimes \Z_p$. Let $E/K$  be such that $j(E)\notin \FF$ and it has good or split multiplicative reductions at all primes of $K$. Then $R(E/K_\infty)^\vee$ is a finitely generated  { $\Z_p[[H]]$} module and in particular, it is  $\Z_p[[G]]$ torsion. 
\end{proposition}
At this point, it is natural to ask if the analogue of Conjecture \ref{conjb} is true in this function field setting (${\ell}\neq p$) and in particular, for $K_\infty$, whether $R(E/K_\infty)^\vee$ is $\Z_p[[H]]$ torsion?
We will give an explicit counterexample (Example \ref{example}) to this question using the following result of \cite{bv}.
\begin{proposition} \cite[Proposition 4.3]{bv} \label{prop5.4}
	Let $E, K$ and $S$ be as defined in \S \ref{section4}.
	Let $G=\gal(\mathcal L/K)$ be a compact $p$-adic Lie group without elements of order $p$ and of dimension $\geq 2$. If $H^2(G_S(\mathcal L), E_{p^\infty})=0$ and $cd_p(G_v)=2$ for every $v\in S$, then $R(E/\mathcal L)^{\vee}$ has no non-trivial pseudonull submodule.
\end{proposition}
Thus to use the above result for $\mathcal L=K_\infty$, we need to establish the vanishing of $H^2(G_S(K_\infty),E_{p^\infty})$.

\begin{proposition}\label{proposition3.13}
	Let $E/K$ be such that $j(E)\notin \FF$. Then  $H^2(G_S(K_\cyc),E_{p^\infty})=0$ and consequently, $H^2(G_S(K_\infty),E_{p^\infty})$ also vanishes.
	
\end{proposition}
\begin{proof}
	As $H\cong \Z_p$, using Hochschild-Serre spectral sequence $H^2(G_S(K_\cyc),E_{p^\infty})=0$ implies $H^2(G_S(K_\infty),E_{p^\infty})=0$. As $R(E/K_\cyc)^\vee$ is a finitely generated $\Z_p$ module, the vanishing of $H^2(G_S(K_\cyc),E_{p^\infty})$ can be deduced form Jannsen's spectral sequence \cite{j2} (see \cite[Lemma 3.1]{cs} for the details).
\end{proof}
Now,  we present a counterexample to an analogue of Conjecture \ref{conjb}.
\begin{example}\label{example}
Set $K=\FF_5(t)$ and let $E$ be the elliptic curve defined over $K$ given by the Weierstrass equation:\begin{small}{\begin{equation*} y^2+xy=x^3-t^{6}.\end{equation*}}
	\end{small}
	Here the discriminant of $E$ { is equal to} $t^6(1+3t^6)=3t^6(t^2-2)(t^2+2t+3) (t^2+3t+3) $. Let $p\neq 5$ be an odd prime and take $q(t)$ { to be equal to} $t(t^2-2)(t^2+2t+3) (t^2+3t+3)$. Now, define $K_\infty= \underset{n\geq 0}\bigcup\FF_5^{(p)}(t) (q(t)^{1/p^n})$. The primes of $K$ that ramify in $K_\infty$ are given by the irreducible divisors of $q(t)$ in $K$ and also the infinite prime. In \cite[\S 2.2 and \S 2.3]{ul}, it was shown that the bad primes for $E$ are the irreducible divisors of $q(t)$ in $K$ and $E$ has split multiplicative reduction at all these bad primes. Hence, we can take $S=\{\text{irreducible divisors of }q(t) \text{ in } K, \text { the infinite prime} \}$. By \cite[Lemma 4.2]{bl2}, we get that $cd_p(G_v)=2$ for every $v\in S$. Also by Proposition \ref{proposition3.13}, we have $H^2(G_S(K_\infty), E_{p^\infty})=0 $. Hence, all the hypotheses of Proposition \ref{prop5.4} are satisfied. 
	Using \cite[Theorem 1.5]{ul}, we also obtain that $\ran_{\Z}(E(K))\geq 2$ and hence $\cora_{\Z_p}(R(E/K))\geq 2$. As $j(E)\notin \FF_5$, from Lemma \ref{lemma3.4}, we deduce that $\cora_{\Z_p}(R(E/K_\infty)^G)\geq 2$. This implies that $R(E/K_\infty)\neq 0$ (Theorem \ref{theo:Ho}).

	Hence, $R(E/K_\infty)^\vee$ is a  torsion $\Z_p[[G]]$ module, which is  finitely generated over $\Z_p[[H]]$  and is not a pseudo null $\Z_p[[G]]$ module.
\end{example}
\section{Fine Selmer group over function fields of characteristic $ p$}\label{section5}

{ We fix a prime $p$ (includes $p=2$)}.
Let $\FF$ be a finite field of order $p^r$, for some $r \in \mathbb N$ and  $K$ denote the function field $\FF(t)$. In this section, we  study the $p^\infty$-fine Selmer groups of elliptic curves over  $p$-adic Lie extensions of $K$.  In fact, for simplicity, we restrict ourselves to the case when $\gal(K_\infty/K)\cong \Z_p^d$. Although, we believe analogue of Conjecture \ref{conjb} should be investigated over non-commutative $p$-adic Lie extensions of function fields of characteristic $p$ also.  We crucially use the properties of the arithmetic (unramified) $\Z_p$ extension and   $\Z_p^{d-1}$ extensions of $K$ constructed using Carlitz modules and hence discuss an analogue of {Conjecture} \ref{conjb} for the composite $\Z_p^d$ extensions of $K$ and prove it under suitable hypotheses.

Set $C_K:=\mathbb P^1_{\mathbb F}$. Throughout this section, $E$ will denote  an  elliptic curve over $K$ and $U$ will be a dense open subset of $C_K$ such that { $E/K$ has good reductions at every place of $U$.} Let $\mathcal{E}$ denote the N\'eron model of $E$ over $C_K$.   Let $\Sigma_K$ be the set of all the places of $K$  and $S_1$ be the  set of places of $K$ outside $U$ i.e., the places of $C_K\setminus U$.  Let $S_2$ be the finite set of primes of $K$, where $E$ has supersingular reduction. { Let $S$ to be the finite set of places of $K$ containing $S_1$ and  $S_2$. We fix the set $S$ throughout \S \ref{section5}.}

Let $L$ be a finite extension   of $K$ inside $K_S$, the maximal algebraic extension of $K$ unramified outside $S$. Let $v$ be any  prime of $K$ and $w$ denote  a prime of $L$. Define
\begin{small}
\begin{equation}\label{equationdefine}
	\begin{aligned}
		J_v^1(E/L):=\underset{w\mid v}\prod	\frac{H^1_{\fla}(L_w, E_{p^\infty})}{im(\kappa_w)} \ \text{and }
		K_v^1(E/L):= \underset{w\mid v}\prod	H^1_{\fla}(L_w, E_{p^\infty}).	
	\end{aligned}
\end{equation}
\end{small}
Here $H^i_{\text{fl}}(-,-) $ denote the flat cohomology \cite[Chapters II, III]{m2} and $\kappa_w: E(L_w)\otimes {\Q_p}/{\Z_p}\inj H^1_{\fla}(L_w,E_{p^\infty})$ is induced by the Kummer map. Recall the following  definition of the Selmer group from \cite{kt}:
\begin{definition}\cite[Prop. 2.4]{kt}
	With $\Sigma_K, S \text{ and } K\subset L \subset K_S$ as above, define 
	\begin{small}
	\begin{equation}\label{definition24}
		S(E/L):= \ker\big(H^1_{\fla}(L, {E}_{p^\infty})\lra \underset{v\in \Sigma_K}\bigoplus J_v^1(E/L)\big).
	\end{equation}
\end{small}
\noindent Analogous to the definition of the fine Selmer group over {a} number field (\ref{defn1.1}) and {a} function field of characteristic $\neq p$  (\ref{def3.1}), we define the $S$-fine Selmer group as:
	\begin{small}
	\begin{equation}\label{definition4.1}
		\begin{aligned}
			R^S(E/L):={} \ker\big(H^1_{\fla}(L, E_{p^\infty})\lra \underset{v\in S}\bigoplus K_v^1(E/L) \underset{v\in \Sigma_K\setminus S}\bigoplus J_v^1(E/L)\big)\\
			\cong \ker\big (S(E/L)\lra \underset{w\mid v, v\in S}\bigoplus E(L_w)\otimes \Q_p/\Z_p\big).
		\end{aligned}
	\end{equation}
\end{small}
\end{definition}
Note that in the number field case for all  places $w\nmid p$ of $L$ and in the function field case of char $\neq p$, for all primes $w$ of $L$, $E(L_w)\otimes \Q_p/\Z_p =0$.

For an infinite algebraic extension $\mathcal L$ of $K$, the above definitions extends,  as usual, by taking inductive limit over finite subextensions of $\mathcal L$ over K. 

\begin{remark}[Dependence on $S$]\label{remark2.2}
	With $E,S,K$ as above, recall the
	following equivalent definition \cite{kt} of the Selmer group:
	\begin{small}
	\begin{equation}\label{definition25}
		S(E/K):=  \ker\big(H^1_{\fla}(U, \mathcal{E}_{p^\infty})\lra \underset{v\in S}\bigoplus J^1_v(E/K)\big ).
	\end{equation}	
\end{small}
	Using definitions \ref{definition4.1} and \ref{definition25}, we have
	\begin{small}
	\begin{equation}\label{defn25}
		R^S(E/K):=  \ker\big(H^1_{\fla}(U, \mathcal{E}_{p^\infty})\lra \underset{v\in S}\bigoplus K^1_v(E/K)\big).
	\end{equation}
\end{small}
	In fact, in \cite[Proposition 2.4]{kt}, the authors showed that the two definitions (\ref{definition24} and \ref{definition25}) of $S(E/K)$  are equivalent.  The key ingredient in the proof is the following exact sequence \cite[Chapter 3, \S 7]{m1}:
	\begin{small}{\begin{equation}
		0\rightarrow H^1_{\fla}(U, \mathcal{E}_{p^\infty})\lra H^1_{\fla}(K, {E}_{p^\infty})\lra \underset{v\in U}\oplus H^1_{\fla}(K_v, E_{p^\infty})/H^1_{\fla}(O_v, E_{p^\infty}),
	\end{equation}}\end{small}
	where $O_v$ is the valuation ring of $K_v$ and $H^1_{\fla}(O_v, E_{p^\infty})\cong E(K_v)\otimes \Q_p/\Z_p .$  
	In particular, it shows that the definition of $S(E/K)$ in \eqref{definition25} is independent of $S$.
	However in the definition of $R^S(E/K)$, for $v\in S$, we have a different local term given by $K_v^1(E/K)$.
	
	For any $p$-adic Lie extension $L_\infty/K$ where $G$= $\gal(L_\infty/K)$  is a compact $p$-adic Lie group   without any $p$-torsion,  $\Z_p[[G]]$ is left and right  Noetherian. For any two finite sets $S_1 \subset S_2$ of primes of $K$ containing the primes of bad reduction of $E$, we have $R^{S_2}(E/L_\infty)\subset  R^{S_1}(E/L_\infty) $. Thus, for a sufficiently large $S$, using the Noetherianess of $\Z_p[[G]]$ and  the module $R^{S_1}(E/L_\infty)^\vee$, we see that $R^S(E/L_\infty)$ is independent of $S$. However, we cannot determine this set $S$ explicitly.  
\end{remark}
Now we discuss  Euler characteristic of $R^S(E/L_\infty)$, where $\mathrm{Gal}(L_\infty/K)\cong\Z_p^d$.
\begin{proposition}\label{thm:control}
	Let $K\subset L_\infty\subset K_S$ be a Galois extension over $K$ such that $\gal(L_\infty/K)\cong \Z_p^d.$ 
Let $E/K$ be an elliptic curve and  $S$ be as defined before. Then $\chi(G,R^S(E/L_\infty))$ exists if $R^S(E/K)$ is finite.
\end{proposition}
\begin{proof}
	Since $G$ is commutative,  it suffices to show that $ R^S(E/L_\infty)^{G} $ is finite.
	Note that $E_{p^\infty}(K)$ is finite. Also, for each $v\in S$, $E_{p^\infty}(K_v)$  is finite  (\cite[Theorem 4.12, \S 2.1.2]{bl1} and \cite[2.5.1]{t}). By our assumption, $R^S(E/K)$ is finite. Notice that the local terms appearing in the definitions of $R^S(E/K)$ and  $S(E/K)$ outside $S$, are the same.
	Now using a control theorem for $R^S(E/K)\lra R^S(E/L_\infty)^{G}$,  similar to the control theorem for $S(E/K)\lra S(E/L_\infty)^{G}$ \cite[Theorem 4.4]{bl1}, we get that $ R^S(E/L_\infty)^{G}$ is finite. 
\end{proof}	
\subsection{Analogues of Conjecture \ref{conja}}\label{subs3.1}
Next, we will discuss  analogues of Conjecture \ref{conja} in this setting. Note that the only $p^\text{th}$-root of unity in $\bar{\FF}_p$ is $1$. We discuss  analogues of Conjecture \ref{conja} for two specific $\Z_p$ extensions  of $K$, widely discussed in the literature; namely, the arithmetic $\Z_p$ extension $K_\infty$ and the geometric $\Z_p$ extension  $\widetilde{K}_1$. Let $\FF_p^{(p)}$ be the unique subfield of $\bar{\FF}_p$ such that $\gal(\FF_p^{(p)}/\FF_p)\cong \Z_p$. Set $K_\infty:=K\FF_p^{(p)}$. Notice that $K_\infty/K$ is unramified everywhere. On the other hand, let $ \widetilde{K}_1$ be { a} geometric $\Z_p$ extension of $K$ constructed using Carlitz module, a particular type  of Drinfeld module.

{Let us briefly recall the construction of $\Z_p^d$ extensions arising from Cartliz-modules (see \cite[Chapter 12]{ro} and also \cite{bl1}). 
Recall that $\FF$ is a finite field of characteristic $p$ and $K=\FF(t)$ is a function field of characteristic $p$.
Let $P(t)=a_nt^n+\cdots + a_0 \in \FF [t]$ be a polynomial of degree $n$ with coefficients in $\FF$. The Carlitz polynomial associated to $P(t)$, denoted by $[P(t)](X)$, is defined  recursively as follows:

		$[1] (X)=X$, 
		
		$[t] (X)=X^p +tX$,
		
		$[t^n] (X)=[t]([t^{n-1}](X))$ and
		
		$[a_n t^n+\cdots+a_1 t+a_0](X)= a_n[t^n](X)+\cdots +a_1[t](X)+a_0 (X).$
		
		 Let $F$ be a field extension  of $K$. Then $F$ can be thought of as a $\FF[t]$-module, where the action of $\FF[t]$  is given by the Carlitz polynomials.

        Choose a non-zero prime ideal  $(\mathfrak P(t))$ of $\FF[t]$, generated by an irreducible polynomial $\mathfrak{P}(t)$. 

        For $n\in \N$, consider the Carlitz polynomial  $ [\mathfrak{P}^n](X)$. Let $\overline{K}$ denote the separable closure of $K$. For $n\geq 1$, define
		$$\Lambda_{\mathfrak P^n}:=\{\lambda\in \overline{K}| [\mathfrak{P}^n](\lambda)=0\},$$

    \noindent  the set of all roots of $[\mathfrak{P}^n](X)$ in $\overline{K}$. 
        It can be verified that $K(\Lambda_{\mathfrak P ^n})/K$ is a Galois extension with 
        Galois  group $\gal(K(\Lambda_{\mathfrak P ^n})/K)\cong (\FF[t]/\mathfrak{P}^n)^\times$ \cite[Theorem 12.7]{ro}. Put $\widetilde{K}=\widetilde{K}_{\mathfrak{P}} :=\underset{n\geq 1}\bigcup K(\Lambda_{\mathfrak{P}^n})$. Then  $\widetilde {K}/K$ is Galois with
$\gal(\widetilde K /K)\cong \Z_p^\N \times (\FF[t]/\mathfrak{P})^\times$.

For each $d\geq 1$, the above mentioned construction of $\widetilde{K}$ via Carlitz module gives rise to Galois extensions  $\widetilde{K}_d$  of $\widetilde{K}$  such that 
$K\subset \widetilde{K}_d\subset \widetilde{K}_{d+1}\subset \widetilde{K}$, $\gal(\widetilde{K}_d/K)\cong \Z_p^d$ and $\gal\big((\underset{d\geq 1}\cup \widetilde{K}_d)/K\big)\cong \Z_p^\N$. 
For each $d\geq 1$, $\widetilde{K}_d$ is ramified only at the prime $\mathfrak P$ and it is totally ramified at that prime \cite[Proposition 12.7]{ro}.}

We will show that, under suitable hypotheses, the analogues of Conjecture \ref{conja} are true over  $K_\infty$ (Theorem \ref{theo:muinvariant}) and also $ \widetilde{K}_{1}$ (Corollary \ref{cormu}) i.e. $R^S(E/K_\infty)^\vee$ and $R^S(E/ \widetilde{K}_{1})^\vee$ are finitely generated $\Z_p$ modules. One of our main tool in proving this is to explore the relation between the fine Selmer group of $E[p]$ with the corresponding divisor class group of extensions of function field.
 {\it For the rest of this section \ref{section5}, we assume that $E/K$ is an ordinary elliptic curve.}
\begin{definition}\label{definition4.4}
	Let $E, K, S$ be as before and  $L \subset K_S$ be a finite extension of $K$. For $*\in \{ \mu_{p^\infty},\ \Q_p/\Z_p,\ \mu_p,\ \Z/p\Z\}$, set \begin{small}{$K_v^1(*/L):$= $\underset{w\mid v}\bigoplus H^1_{\fla}(L_w, *)$ and $J_v^1(*/L):=\underset{w\mid v}\bigoplus H^1_{\fla}(L_w, *)/H^1_{\fla}(O_w,*)$.}\end{small} We define the groups $S'(*/L)$ and $R^S(*/L)$  as follows:
	\begin{small}{
	\begin{equation}\label{s'defn}
		S'(*/L) := \ker(H^1_{\fla}(L, *)\lra \underset{v\in \Sigma_K}\bigoplus J_v^1(*/L)).
	\end{equation}
}\end{small}
\begin{small}{
	\begin{equation}\label{equation27}
		R^S(*/L):=\ker(H^1_{\fla}(L, *)\lra \underset{v\in S}\bigoplus K_v^1(*/L) )\underset{v\in \Sigma_K\setminus S}\bigoplus J_v^1(*/L)),
	\end{equation}}
\end{small}
	Similarly, the  group $R^S(E[p]/L)$ of $E[p]$ over  $L$ is defined by:
	\begin{small}{\begin{equation}\label{equation27}
		R^S(E[p]/L):=\ker(H^1_{\fla}(L, E[p])\lra \underset{v\in S}\bigoplus K_v^1(E[p]/L) )\underset{v\in \Sigma_K\setminus S}\bigoplus J_v^1(E[p]/L)),
	\end{equation}
where $K_v^1(E[p]/L):= \underset{w\mid v}\bigoplus H^1_{\fla}(L_w, E[p])$ and $J_v^1(E[p]/L):=\underset{w\mid v}\bigoplus H^1_{\fla}(L_w, E[p])/H^1_{\fla}(O_w,\mathcal E[p])$.}\end{small}
\end{definition}
At first, note that by \cite[Theorem 1.7]{ot}, $S(E/K_\infty)^\vee$ (and hence $R^S(E/K_\infty)^\vee$) is a finitely generated torsion $\Z_p[[\Gamma]]$ module. 
\begin{lemma}\label{lemma4.10}
 $\mu(R^S(E/K_\infty)^\vee) =0$ if and only if $R^S(E[p]/K_\infty)$ is finite. 
\end{lemma}
\begin{proof}
	It is easy to see that the kernel and the cokernel of the natural map                    $R^S(E[p]/K_\infty) \rightarrow R^S(E/K_\infty)[p]$ are finite. The result follows from this. 
\end{proof}
\begin{proposition} \label{prop4.13}
Define $K_\infty^{p}: =K\bar{\FF}_p$ and recall $E/K$ is ordinary. Then  $R^S(E[p]/K_\infty^{p})$ is finite.
\end{proposition}
\begin{proof}

	Consider the connected-étale sequence (see, for example,  \cite[\textsection 3.2]{llstt})
	\begin{small}
	\begin{equation}\label{equn4.11}
		0\lra E[p]^0\lra E[p]\lra \pi_0(E[p])\lra 0
	\end{equation}
\end{small}
	where $E[p]^0$ and $\pi_0(E[p])$ are Cartier dual to each other.

    { For a finite   Galois extension $M$ of $K$, put $G_M:=\gal(\bar{M}/M)$, where $\bar{M}$ is the separable closure of $M$.}
    Let $L$ be a finite Galois extension of $K$ such that the action of $G_L$  on $\pi_0(E[p])$ is trivial.
 Note that $E/L$ is ordinary. Therefore, we know that $\pi_0(E[p])\cong \Z/p\Z$ and $E[p]^0\cong \mu_p$, where $\mu_p$ is the Cartier dual to $\Z/p\Z$.

    {Let  $L_\infty^{p}: =L\bar{\FF}_p$. }
	Hence, we have the following exact sequences,
	\begin{small}
	\begin{equation}\label{equation29}
		H^1_{\fla}(L_\infty^{p},\mu_p)\lra H^1_{\fla}(L_\infty^{p},E[p])\lra H^1_{\fla}(L_\infty^{p},\Z/p\Z).
	\end{equation}
\end{small}
	Also for each prime $w$ of $K_\infty$,
	\begin{equation}\label{equation30}
		H^1_{\fla}(L_{\infty,w}^{p},\mu_p)\lra H^1_{\fla}(L_{\infty,w}^{p},E[p])\lra H^1_{\fla}(L_{\infty,w}^{p},\Z/p\Z).
	\end{equation}
	By \cite[\textsection 1.4]{ta}, for $w$ dividing $v$ of $L$, {$v\nmid S$}, we have the following exact sequence:
	\begin{small}
	\begin{equation}
		0\lra \mathcal E(O_w)[p]^0\lra \mathcal E(O_w)[p]\lra \pi_0(\mathcal E(O_w)[p])\lra 0,
	\end{equation}
\end{small}
	where $\mathcal E(O_w)[p]^0\cong \mu_p $ and $ \pi_0(\mathcal E(O_w)[p])\cong \Z/p\Z$. 
	
	From the definition in \eqref{equation27}, we have the following commutative diagram of complexes, which is not necessarily exact:
	\begin{small}
		\begin{equation}\label{equation34}
			\begin{tikzcd}
				0\arrow{r}
				&R^S\big((\Z/p\Z)/L_\infty^{p} \big)\arrow{r}
				&H^1_{fl}(L_\infty^{p}, \Z/pZ)\arrow{r}
				&\underset{w\mid v, v \in S} \prod H^1_{\fla}(L_{\infty,w}^{p},\Z/p\Z) \underset{w\mid v, v \notin S}\prod \frac {H^1_{\fla}(L_{\infty,w}^{p},\Z/p\Z)} {H^1_{\fla}(O_{\infty,w},\Z/p\Z)}
				\\
				0\arrow{r}
				&R^S(E[p]/L_\infty^{p})\arrow{r}\arrow{u}{\alpha}
				&H^1_{fl}(L_\infty^{p}, E[p]) \arrow{r}\arrow{u}
				&\underset{w\mid v, v \in S}\prod H^1_{\fla}(L_{\infty,w}^{p},E[p])\underset{w\mid v, v \notin S}\prod \frac {H^1_{\fla}(L_{\infty,w}^{p},E[p])} {H^1_{\fla}(O_{\infty, w},\mathcal E[p])}\arrow{u}
				\\
				0\arrow{r}
				&	R^S(\mu_p/L_\infty^{p})\arrow{r}\arrow{u}{\beta}
				& H^1_{\fla}(L_\infty^{p},\mu_p)	\arrow{r}\arrow{u}
				&\underset{w\mid v, v \in S} \prod H^1_{\fla}(L_{\infty,w}^{p},\mu_p) \underset{w\mid v, v \notin S}\prod \frac {H^1_{\fla}(L_{\infty,w}^{p},\mu_p)} {H^1_{\fla}(O_{\infty, w},\mu_p)}\arrow{u}{\gamma=\underset{w}\prod\gamma_w}
			\end{tikzcd}
		\end{equation}
	\end{small}
	Note that 	$R^S(\mu_p/ L_\infty^{p})\inj S'(\mu_p/L_\infty^{p})$ and ${R^S((\Z/p\Z)/ L_\infty^p)}\inj S'\big((\Z/p\Z)/L_\infty^{p}\big) $, where $S'\big(-/L_\infty^{p}\big) $ is defined in \eqref{s'defn}. By the proof of \cite[Lemma 3.4]{ot}, we deduce that $S'\big((\Z/p\Z)/L_\infty^{p}\big)$ is finite and  $S'(\mu_p/L_\infty^{p})\cong Cl(L_\infty^{p})[p]$, the $p$-part of the divisor class group, which is also finite \cite[Proposition 11.16]{ro}. (Notice that for $*\in\{\mu_p,\ \Z/p\Z\}$, $S'(*/L_\infty^p)$ is being denoted by $H^1_{fl}(C_{L_\infty^p},*) $ in \cite{ot}). Hence, $R^S(\mu_p, L_\infty^{p})$ and ${R^S((\Z/p\Z)/ L_\infty^p)}$ are finite. 
	Now, using \cite[\S III.7]{m1}, we get that  for $w\mid v, v\notin S$, $\ker(\gamma_w)=0$. Also, for $w\mid v, v\in S$, $\ker(\gamma_w)$ is finite. This implies that $\ker(\gamma)$ is finite. By applying a snake lemma to the lower complex in \eqref{equation34}, we obtain a map from  $\ker(\gamma)\overset{\theta}\lra \text{coker}(\beta)$, such that coker$(\theta)$ injects into the finite group $R^S\big((\Z/p\Z)/L_\infty^{p}   \big)$.
	Therefore,  $R^S(E[p]/L_\infty^{p}) $ is finite, { which in turn shows that $R^S(E[p]/K_\infty^{p}) $ is finite as well. }
\end{proof}

As $G=\gal(K_\infty^p/K_\infty)\cong \underset{l\neq p}\prod \Z_l$,  we get $R^S(E[p]/K_\infty^p)^{G}\cong R^S(E[p]/K_\infty) $. Thus using Lemma \ref{lemma4.10} and Proposition \ref{prop4.13}, we
deduce the following theorem:
\begin{theorem}\label{theo:muinvariant}
	Let $E/K$ be an  ordinary elliptic curve and  $S$ be as defined in \S \ref{section5}.  Then, $R^S(E/K_\infty)^\vee$ is a finitely generated $\Z_p$ module. \qed
\end{theorem}
 \subsection{Pseudonullity}\label{sec4.2} Now we discuss an analogue of Conjecture \ref{conjb} over the extension $F_\infty:=K_\infty  \widetilde{K}_{1}$ of $K$. Put $G$=$\mathrm{Gal}(F_\infty/K) \cong \Z_p^2$ and $H:$=$\mathrm{Gal}(F_\infty/K_\infty) \cong \Z_p$. We 
will show, under suitable assumption $R^S(E/F_\infty)^\vee$ is a pseudonull $\Z_p[[G]]$ module. We begin by collecting some evidence towards this. At first, in Proposition \ref{theo5.9}, we show that the corank of $R^S(E/K)$ is strictly less than the corank of $S(E/K)$, whenever $E(K)$ is infinite. 
\begin{proposition}\label{theo5.9}
	Let $E/K$ be an elliptic curve with $\ran_{\Z} (E(K)) \geq1$. Assume that there exists   $v\in S$, where $E$ has good ordinary reduction or split multiplicative reduction. Then, $\cora_{\Z_p}(R^S(E/K))< \cora_{\Z_p}(S(E/K))$.
\end{proposition}
\begin{proof}
	We generalise the proof of \cite[Lemma 4.1]{cs} from the number field case to the function field (of char. $p$) case.  
	Recall the following short  exact sequence,\noindent
\begin{small}{
	\begin{equation}
		0 \lra R^S(E/K) \lra S(E/K)\overset{\underset{v\in S}\oplus r_v} \lra  \underset{v\in S}{\bigoplus} E(K_v)\otimes \Q_p/\Z_p.
	\end{equation}}
\end{small}
So, it is enough to prove that $\text{image} (\underset{v\in S}\oplus r_v)$ is infinite. Note that $B=E(K)\otimes \Q_p/\Z_p\inj S(E/K)$. Since $E(K)$ is finitely generated,  $B\cong (\Q_p/\Z_p)^{\ran(E(K))}$ is divisible. Thus it is enough to show that $r_{v}(B)\neq 0$, for some $v\in S$.
	On the contrary, let us assume that $r_{v}(B)= 0$ for every $v\in S$ and let $R$ be a point of infinite order in $E(K)$. Then for every $k\geq 1$, $R\otimes p^{-k}=0$ in $E(K_v)\otimes \Q_p/\Z_p$.

	Let us consider the case where $E$  has a split multiplicative reduction at $v\in S$. Let $E_0(K_v)$ be the set of points of $E(K_v)$ which has non-singular reduction at the residue field of $K_v$. We can choose an integer $n\neq 0$ such that $Q=nR$ is a point on $E_0(K_{v})$. Then, for every $k\geq 1$, $Q\otimes p^{-k}=0$ in $E_0(K_v)\otimes \Q_p/\Z_p$.  Note that $E_0(K_{v})\cong O_v^*$ \cite[\S 2.1.2]{bl1}, where $O_v$ is the ring of integers of $K_{v}$. This is a contradiction since $O_v^*\cong \Z_p^\N\oplus \Z/(p^r-1)\Z $, where $p^r=\#\FF$. 
	
	Next, let us assume that $v$ is a prime of good ordinary reduction. We can choose an integer $n\neq 0$ such that $Q=nR$ is a point on $E_{v}(m_{v})^\vee$, the formal group of $E$ at $v$. Note that $E_{v}(m_{v})^\vee $  is a torsion free $\Z_p$ module  \cite[Lemma 2.5.1]{t}. Being a finite index subgroup of the maximal pro-$p$ subgroup of $E(K_{v})$, $E_{v}(m_{v})^\vee $ is a pro-$p$ group \cite[page 4436]{t}. By \cite[Theorem 4.3.4]{rz}, we know that a torsion free pro-$p$ group is a free $\Z_p$ module. Therefore, $E_{v}(m_{v})^\vee$ is a free $\Z_p$ module and it cannot have a infinite $p$-divisible element.
\end{proof}
\noindent From Proposition \ref{theo5.9}, we get an evidence for Conjecture B over $\Z_p^d$-extension.
\begin{corollary}\label{theo4.8}
Let us keep the  hypotheses of  Proposition \ref{theo5.9} and assume that $j(E)\notin \FF$. Let $K \subset L_\infty \subset K_S$ be   such that $\gal(L_\infty/K)\cong \Z_p^d$. Then $S(E/L_\infty) \neq R^S(E/L_\infty)$.
\end{corollary}
\begin{proof}
	We use Proposition \ref{theo5.9} and using a control theorem, proceed in a 
	similar way,  as in the proof of  \cite[Proposition 4.3]{cs}. The details are omitted.
\end{proof}
\noindent For an algebraic extension $L/K$, let $Cl(L)$ denote the divisor class group of $L$. 

\begin{proposition}\label{theorem4.14}
	Recall $F_\infty=K_\infty  \widetilde{K}_1$ and  $H=\gal(F_\infty/ \widetilde{K}_1) \cong \Z_p$. Then, $Cl(F_\infty)[p^\infty]^{\vee}$ is a  finitely generated torsion $\Z_p[[H]]$ module.
\end{proposition} 
\begin{proof}
 By Nakayama lemma, it suffices to show that $(Cl(F_{\infty})[p^\infty])^{H}$ is finite. 
 
 Using the commutative diagram in \cite[Page-38]{ot}, we obtain the following commutative diagram:	\noindent
	\begin{small}{
	\begin{equation}\label{equation35}
		\begin{tikzcd}
			0\arrow{r}
			&\left(\underset{m}\ilim\ \dfrac{(\FF^{(p)})^\times}{((\FF^{(p)})^\times)^{p^m}}\right)^{H}\arrow{r}
			&S'(\mu_{p^\infty}/F_\infty)^{H}\arrow{r}{\beta}
			&(Cl(F_{\infty})[p^\infty]) ^{H}
			\\
			0\arrow{r}
			&\underset{m}\ilim\ \dfrac{\FF^\times}{(\FF^\times)^{p^m}}\arrow{r}\arrow{u}
			&S'(\mu_{p^\infty}/ \widetilde{K}_1)\arrow{r}\arrow{u}{\alpha}
			&Cl( \widetilde{K}_1)[p^\infty]\arrow{r} \arrow{u}{\gamma}
			&0	
		\end{tikzcd}
	\end{equation}}
\end{small}
	\noindent Note that $\dfrac{(\FF^{(p)})^\times}{((\FF^{(p)})^\times)^{p^m}}=0$ for all $m$. Therefore, $\beta$ is an isomorphism.  Further, by \cite[Proposition 2]{a},  $Cl( \widetilde{K}_1)[p^\infty] $ is finite. We claim that $\text{coker}(\alpha)$ is finite. Then  using a snake lemma in diagram \ref{equation35}, $\ker(\gamma)$ is finite and the lemma follows. We now establish the claim.  Consider the  commutative diagram:
	\begin{small}
	\begin{equation}\label{diagram36}
		\begin{tikzcd}
			0\arrow{r}
			&S'(\mu_{p^\infty}/F_\infty)^{H}\arrow{r}
			& H^1_{\fla}(F_{\infty},\mu_{p^\infty})^{H}\arrow{r}
			&\left(\underset{w\mid v, v \in \Sigma_K}\oplus \frac {H^1_{\fla}(F_{\infty,w},\mu_{p^\infty})} {H^1_{\fla}(O_{F_{\infty, w}},\mu_{p^\infty})}\right)^H
			\\
			0\arrow{r}
			&S'(\mu_{p^\infty}/ \widetilde{K}_1)\arrow{u}{\alpha}\arrow{r}
			& H^1_{\fla}( \widetilde{K}_1,\mu_{p^\infty})\arrow{u}{\delta}\arrow{r}
			&\left(\underset{w\mid v, v \in \Sigma_K}\oplus \frac {H^1_{\fla}( \widetilde{K}_{1,w},\mu_{p^\infty})} {H^1_{\fla}(O_{ \widetilde{K}_{1,w} },\mu_{p^\infty})}\right)\arrow{u}{\eta}
		\end{tikzcd}
	\end{equation}
\end{small}
	Using the relation: $H^1_{\fla}(\text{Spec}(R), \mu_{p^n})=R^\times/(R^\times)^{p^n}$, for a local ring $R$ and the Hochschild-Serre spectral sequence, we obtain that $\ker(\eta)=0$. As $H\cong\Z_p$
, coker$(\delta)=0$. Then $\text{coker}(\alpha)=0$ follows from the diagram \eqref{diagram36}.\noindent \end{proof}
\begin{lemma}\label{lemma4.11}
 $R^S\big((\Z/p\Z)/F_\infty\big)$ is finite.
\end{lemma}
\begin{proof}
	
	We have $F_\infty=\underset{n,m}\bigcup F_{n,m}$, where $\gal(F_{n,m}/F)\cong \Z/p^n\Z\times \Z/p^m\Z$. Using \cite[Theorem 27.6]{st}, we observe that \begin{small}$R^S\big((\Z/p\Z)/F_{n,m} \big)\cong \Hom(G_{S}^{\text{ab}}(F_{n,m})(p),\Z/p\Z)\inj \Hom(G_{\phi}^{\text{ab}}(F_{n,m})(p), \Z/p\Z)$\end{small}, where $G_{\phi}(F_{n,m})^{\text{ab}}(p)$ is the Galois group of the maximal abelian everywhere unramified pro-$p$ extension of $F_{n,m}$. Also,\\ $\Hom(G_{\phi}^{\text{ab}}(F_{n,m})(p), \Z/p\Z)\cong \Hom(Cl(F_{n,m})\otimes\Z_p, \Z/p\Z)\cong \Hom(Cl(F_{n,m})/p, \Z/p\Z)$. By applying a  control theorem, similar to the proof of Proposition \ref{theorem4.14}, we get that the kernel and the cokernel of the map  $Cl(F_{n,m})[p]\lra Cl(F_\infty)[p]^{\gal(F_\infty/F_{n,m})}$ are finite and bounded independently of $m$ and $n$.  As a result, $Cl(F_{n,m})[p]$ is finite and bounded independently of $m$ and $n$. Moreover, as $Cl(F_{n,m})(p)$ is finite, the same is true for $Cl(F_{n,m})/(p) $. Thus, $R^S\big((\Z/p\Z)/F_{n,m}\big)$ is  finite and bounded independently of $n$ and $m$. Hence, we conclude that $R^S\big((\Z/p\Z)/F_\infty\big)\cong \underset{n}\ilim\ R^S\big((\Z/p\Z)/F_{n,m}\big)$ is finite.
\end{proof}
\begin{lemma}\label{lemm4.12}
 $S'\big(\mu_p/F_\infty\big)$ is finite. It follows that $R^S(\mu_p/F_\infty)$ is also finite. 
\end{lemma}
\begin{proof}
	We have an exact sequence:
    
$0\lra \dfrac{(\FF^{(p)})^\times}{((\FF^{(p)})^\times)^{p}}\lra S'(\mu_{p}/F_\infty)\lra Cl(F_{\infty})[p]\lra 0.$
    
    Note that $\frac{(\FF^{(p)})^\times}{((\FF^{(p)})^\times)^{p^m}}=0$, hence $S'\big(\mu_p/F_\infty\big)\cong Cl(F_\infty)[p]$.  We have $ \widetilde{K}_1=\underset{n}\bigcup K_{n}'$, where $K\subset K_n'\subset  \widetilde{K}_1$ with $\gal (K_{n}'/K)\cong\Z/p^n\Z$. Set $K_{n,\infty}=K_{n}'K_\infty$ and $\mathcal G_n=\gal(K_{n,\infty}^{ {p}}/K_{n,\infty})$ with $K_{n,\infty}^{ {p}}=K_{n}'\bar{\FF}_p(t)$. As profinite order of $\mathcal G_n$ is prime to $p$, following a standard diagram chase using the definition of $S'(\mu_{p^\infty}/-)$, we get  $S'(\mu_{p^\infty}/K_{n,\infty}^{ {p}})^{\mathcal{G}_n}\cong S'(\mu_{p^\infty}/K_{n,\infty})$. Similarly, we obtain that $Cl(K_{n,\infty})[p^\infty]\cong {Cl(K_{n,\infty}^{ {p}})[p^\infty]}^{\mathcal G_n}$.
	Let $C_{K_{n,\infty}^{ {p}}}$ be a proper smooth geometrically connected curve which is the model of the function field $K_{n,\infty}^{ {p}}$. On the other hand, \cite[Proposition 10.1.1]{nsw}, we observe that  $Cl(K_{n,\infty}^{ {p}})[p^\infty]\cong (\Q_p/\Z_p)^{r_n}$, where $0\leq r_n\leq \text{genus} (C_{K_{n,\infty}^{ {p}}})$.  Therefore $Cl(F_\infty)[p^\infty]\cong\underset{n}\ilim\ Cl(K_{n,\infty}^{p})[p^\infty]$ is $p$-divisible. By Proposition \ref{theorem4.14}, $Cl(F_\infty)[p^\infty]^{\vee}$ is also a  finitely generated torsion $\Z_p[[H]]\cong\Z_p[[T]]$ module. Consequently, $S'\big(\mu_p/F_\infty\big)\cong Cl(F_\infty)[p]$ is finite. As $R^S(\mu_p/F_\infty)\inj S'\big(\mu_p/F_\infty\big)$ (see Definition  \ref{definition4.4}), $R^S(\mu_p/F_\infty)$ is finite as well.
\end{proof}
\begin{proposition}\label{proposition4.13}
	Recall  $H=\gal(F_\infty/K_\infty)$. Then $R^S(E/F_\infty)^\vee$ is a finitely generated  $\Z_p[[H]]$ module and in particular, it is  $\Z_p[[G]]$ torsion. 
\end{proposition}
\begin{proof}By Theorem \ref{theo:muinvariant}, $R^S(E/K_\infty)^\vee$ is a finitely generated $\Z_p$ module. It is easy to  see that the kernel and the cokernel of the map $(R^S(E/F_\infty)^\vee)_H\lra R^S(E/K_{\infty})^\vee$ are finitely generated $\Z_p$-modules. Hence by Nakayama lemma, $R^S(E/F_\infty)^\vee$ is a finitely generated  $\Z_p[[H]]$ module.
\end{proof}
\begin{theorem}\label{pseudonullity}
	Let $F_\infty=K_\infty  \widetilde{K}_1$ be the $\Z_p^2$ extension of $K$,  defined in \S\ref{sec4.2} and let 
	$v_{r}$ be the unique prime of $K$ that ramifies in  $\widetilde{K}_1$. Assume $E/K$ be an ordinary elliptic curve which  has good{, ordinary} reduction outside $\{v_r\}$ and $S$ be as specified before. Then $R^S(E/F_\infty)^\vee$ is a pseudonull $\Z_p[[G]]$ module.
\end{theorem}

\begin{proof}
Assume that the action of $G_K$ on $\pi_0(E[p])$ is trivial.
Put $S':$=$\{v_r\}$. 
    Using ordinarity of $E/K$ { and the fact that $E$ has ordinary reduction at every prime except (possibly) at $v_r$}, from \eqref{equn4.11}, we obtain a complex (not necessarily exact):
\begin{small}{\begin{equation}R^{S'}(\mu_p/F_\infty)\lra R^{S'}(E[p]/ F_\infty)\lra R^{S'}\big((\Z/p\Z)/ F_\infty\big).\end{equation}}
\end{small}	
Then,  from the definition of $ R^{S'}(\textunderscore /F_\infty)$, there is a commutative diagram (not necessarily exact):
	\begin{small}
		\begin{equation}
			\begin{tikzcd}
				0\arrow{r}
				&R^{S'}\big((\Z/p\Z)/F_\infty\big)\arrow{r}
				&H^1_{\fla}(F_\infty, \Z/p\Z)\arrow{r}
				&\underset{w\mid v_r} \prod H^1_{\fla}(F_{\infty,w},\Z/p\Z) \underset{w\nmid v_r }\prod \frac {H^1_{\fla}(F_{\infty,w},\Z/p\Z)} {H^1_{\fla}(O_{F_{\infty,w}},\Z/p\Z)}
				\\
				0\arrow{r}
				&R^{S'}(E[p]/F_\infty)\arrow{r}\arrow{u}{\alpha}
				&H^1_{\fla}(F_\infty, E[p]) \arrow{r}\arrow{u}
				&\underset{w\mid v_r}\prod H^1_{\fla}(F_{\infty,w},E[p])\underset{w\nmid v_r}\prod \frac {H^1_{\fla}(F_{\infty,w},E[p])} {H^1_{\fla}(O_{F_{\infty, w}},\mathcal E[p])}\arrow{u}
				\\
				0\arrow{r}
				&	R^{S'}(\mu_p/F_\infty)\arrow{r}\arrow{u}{\beta}
				& H^1_{\fla}(F_\infty,\mu_p)	\arrow{r}\arrow{u}
				&\underset{w\mid v_r} \prod H^1_{\fla}(F_{\infty,w},\mu_p) \underset{w\nmid v_r}\prod \frac {H^1_{\fla}(F_{\infty,w},\mu_p)} {H^1_{\fla}(O_{F_{\infty, w}},\mu_p)}\arrow{u}{\gamma=\underset{w}\prod\gamma_w}
			\end{tikzcd}
		\end{equation}
	\end{small}
Here $w$ denote a place of $F_\infty$. Note that $R^{S'}(\mu_p/F_\infty)$ and $R^{S'}\big((\Z/p\Z)/F_\infty\big)$ are finite by Lemmas \ref{lemm4.12} and \ref{lemma4.11}, respectively.   Note that there are only finitely many primes of $F_\infty$ lying above  $v_r$ in $F_\infty$. Now, from the proof of Proposition \ref{prop4.13}, we get that $\ker(\gamma)$ is finite. Again using a diagram  chase, as in the proof of Proposition \ref{prop4.13}, we get that $R^{S'}(E[p]/F_\infty)$ is finite.

    Now, if the action of $G_K$ on $\pi_0(E[p])$ is non-trivial then we move to a finite extension of $L$ of $K$, where the action of $G_L$ on $\pi_0(E[p])$ is trivial. Then arguing as above, we show that $R^{S'}(E[p]/LF_\infty)$ is finite, which in turn shows that $R^{S'}(E[p]/F_\infty)$ is finite as well.
	
	Next, it is easy to see that the kernel and the cokernel of the natural map  \begin{small}$R^{S'}(E[p]/F_\infty)\rightarrow R^{S'}(E/F_\infty)[p]$ \end{small} are finite. Thus \begin{small}$R^{S'}(E/F_\infty)^\vee/(pR^{S'}(E/F_\infty)^\vee)$\end{small} is finite. Then, by  Nakayama lemma and Proposition \ref{proposition4.13},  $R^{S'}(E/F_\infty)^\vee$ is a finitely generated torsion $\Z_p[[H]]$ module and hence  $\Z_p[[G]]$  pseudonull (Theorem \ref{theo:pseudo}).
	
 Now for a finite set $S$ of primes of $K$ containing $S'$=$\{v_r\}$, such that $E/K$ has good reduction outside $v_r$, $R^{S}(E/F_\infty)\inj R^{S'}(E/F_\infty)$. Consequently, $R^S(E/F_\infty)^\vee$ is a pseudonull $\Z_p[[G]]$ module.
\end{proof}
\begin{corollary}\label{cormu}
Let us keep the setting and hypotheses of Theorem \ref{pseudonullity}. Then $R^S(E/ \widetilde{K}_1)^\vee$ is a finitely generated $\Z_p$ module.
\end{corollary}
\begin{proof}
	Again, put $S'=\{v_r\}$.
	Applying a control theorem $R^{S'}(E[p]/ \widetilde{K}_1)\lra R^{S'}(E[p]/F_\infty)^{\gal(F_\infty/ \widetilde{K}_1)}$ and using Theorem \ref{pseudonullity}, we deduce that $R^{S'}(E[p]/ \widetilde{K}_1)$ is finite.  Similarly,  the kernel and the cokernel of the natural map  $R^{S'}(E[p]/ \widetilde{K}_1)\lra R^{S'}(E/ \widetilde{K}_1)[p]$ are finite and thus  $R^{S'}(E/ \widetilde{K}_1)^\vee/(pR^{S'}(E/ \widetilde{K}_1)^\vee)$ is finite. Then by Nakayama lemma, $R^{S'}(E/ \widetilde{K}_1)^\vee$ is  a finitely generated $\Z_p$ module. Finally, for a general $S$, the argument extends as in the proof of Theorem \ref{pseudonullity}.
\end{proof}
{
\begin{remark}
\begin{enumerate}
\item Let $p=2$ and $K=\FF_2(t)$.  Let $E/K$ be given by the Weierstrass equation:
	\begin{small}
		\begin{equation}
			y^2 + xy = x^3 + (1/t) x^2 + 1.
		\end{equation}	
	\end{small}
\noindent Then $E$ has bad reduction only at the prime $(t)$ of $K$ and is an ordinary elliptic curve over $K$. From \cite[\S 5.4]{si}, it is easy to see that $E$ has good, ordinary reduction at all primes of $K$, except the prime $(t)$.
\item Let $K=\FF(t)$ be  a function field over the finite field $\FF$ of char $p$. Let $E/\FF$ be an ordinary elliptic curve. These elliptic curves have good ordinary reduction at all primes of $K$.
\end{enumerate}
    In both these examples all the hypotheses of Theorem \ref{pseudonullity} are satisfied. Hence, for both the examples,  $R^S(E/F_\infty)^\vee$ is a pseudonull $\Z_p[[G]]$ module.
\end{remark}
 }

{\bf Zero  Selmer group:} In this function field of characteristic $p$ setting, we discuss `$p^\infty$-zero  Selmer group' or `$\Sh^1(E_{p^\infty}/L)$' of an elliptic curve $E$ defined over an algebraic extension $L/K$ and compare it with $R^S(E_{p^\infty}/L)$.
\begin{definition}
	Define the `$p^\infty$-zero  Selmer group' of  $E$ over a finite extension $L/K$ as:
	\begin{small}{
\begin{equation*}
	R_0(E/L):= \ker\big(H^1_{\fla}(L, E_{p^\infty})\lra \underset{w\in \Sigma_L}{\prod} K_v^1(E/L)\big).
	\end{equation*}}\end{small}
Let $K\subset L_\infty\subset K_S$ be an infinite extension.
		We define \begin{small}{$R_0(E/L_\infty):= 
		\underset{K\subset F\subset L_\infty}\ilim\ R_0(E/F)$}\end{small}, where $F$ varies over finite extensions of $K$.
\end{definition}
In the following remark, we discuss various properties of  $R_0(E/L)$.
\begin{remark}\label{remark0selmer}
\begin{enumerate}[(i)]
	
		\item Recall $S$ is a finite set of primes of $K$ containing all the primes of bad reduction of $E/K$. Then for any algebraic extension $K\subset L_\infty \subset K_S$, clearly $R_0(E/L_\infty)\inj R^S(E/L_\infty)$, for any such $S$.

	Hence, under the setting of Theorem \ref{theo:muinvariant} (respectively Corollary \ref{cormu}),		
		 we can deduce from the same theorem (resp. corollary) that $R_0(E/K_\infty)^\vee$   (respectively $R_0(E/ \widetilde{K}_1)^\vee$) is a finitely generated $\Z_p$ module. Similarly, from  Theorem  \ref{pseudonullity}, we obtain that $R_0(E/F_\infty)^\vee$ is a pseudo-null $\Z_p[[G]]$ module.
	 \item  For a finite extension  $L/K$, the $p^n$-zero Selmer group is defined as:
	  \begin{equation*}
    R_0(E[p^n]/L):= \ker\big(H^1_{\fla}(L, E[p^n])\lra \underset{w\in \Sigma_L}\prod  H^1(L_w, E[p^n])\big).
\end{equation*}
		In fact, $R_0(E[p^n]/L)= 0$ in most cases (see \cite[Proposition 6.1]{bklpr} for a precise statement).
\end{enumerate}
\end{remark}
\subsection{Pseudonullity in $\Z_p^d$ extensions.}
Let $r\geq 1$. We consider $\Z_p^r$ extensions  $\widetilde{K}_r$ of $K$,
 constructed using Carlitz modules \cite[Chapter-12]{ro}, which are ramified only at one prime  of $K$ and it is totally ramified at that prime \cite[Proposition 12.7]{ro}. 
Consider the following set up:
\begin{equation}\label{diag:fieldextensions}
\begin{tiny}
    \begin{tikzcd}
	& {\mathcal L_\infty} \\
	\\
	&&& { \widetilde{K}_{d-1}} \\
	{F_\infty} \\
	&& { \widetilde{K}_1} \\
	{K_\infty} \\
        \\
	&& K
	\arrow[draw=none, from=1-2, to=8-3]
	\arrow[no head, from=3-4, to=1-2]
	\arrow[no head, from=4-1, to=1-2]
	\arrow[no head, from=6-1, to=4-1]
	\arrow[no head, from=8-3, to=3-4]
	\arrow[no head, from=8-3, to=6-1]
	\arrow[no head, from=5-3, to=4-1]
	\arrow[no head, from=5-3, to=3-4]
	\arrow[no head, from=8-3, to=5-3]
    \arrow[no head, "H'", bend right=30, from=1-2, to=4-1]
    \arrow[no head, "H", bend right=60, from=1-2, to=6-1]
\end{tikzcd}
\end{tiny}
\end{equation}

  Set $d\geq 3$. Let $K_\infty$ denote the unramified $\Z_p$ extension. Consider a $\Z_p^{d-1}$ extension $ \widetilde{K}_{d-1}$ of $K$ constructed from Carlitz modules such that the extension $ \widetilde{K}_{d-1}/K$  is unramified at all but one  prime ideal $\mathfrak{P}$ of $K$. Further the prime ideal $\mathfrak{P}$  is totally ramified in $ \widetilde{K}_{d-1}/K$. Put
$\mathcal{L}_\infty:=K_\infty  \widetilde{K}_{d-1}$.
As in the previous subsection, $F_\infty$ denotes a $\Z_p^2$ extension of $K$ which is the compositum of the uramified $\Z_p$ extension $K_\infty$ of $K$ and a $\Z_p$ extension $ \widetilde{K}_1 $ of $K$ contained in $ \widetilde{K}_{d-1}$.
 Set    
$G:=\gal(\calL_\infty/K)\cong \Z_p^d$, $H=\gal(\calL_\infty/K_\infty)\cong \Z_p^{d-1}$ and $H'=\gal(\calL_\infty/F_\infty)\cong \Z_p^{d-2}$. This setup is presented in the diagram \eqref{diag:fieldextensions}.

\begin{theorem}\label{theo:pseudomain3}
Let $d>2$ and recall $K=\FF(t)$, where $\FF$ is a finite field of char $p$. Let
$\calL_\infty=K_\infty  \widetilde{K}_{d-1}$ be a $\Z_p^d$ extension of $K$ as considered in the diagram \eqref{diag:fieldextensions}. Put $G:=\gal(\calL_\infty/K) $ and let $\nu_{ram}$ be the unique prime of $K$ that ramifies in $\calL_\infty$. Consider a finite set  $S$ of primes of $K$ containing  $\nu_{ram}$. 
	
	Assume that $E/K$ is an ordinary elliptic curve that has good, ordinary reduction outside $\nu_{ram}$. Then $R^S(E/\calL_\infty)^\vee$ is a pseudonull  $\Z_p[[G]]$ module.
\end{theorem}
\begin{proof}
Let $H=\gal(\calL_\infty/K_\infty)$ and $H'=\gal(\calL_\infty/F_\infty)$. Then $H/H'$ is a $p$-adic Lie group of dimension 1. Now, we outline the key steps in the proof.
\begin{enumerate}
    \item First we show that  $R^S(\mu_p/\calL_\infty)^\vee$ is a finitely generated $\Z/p\Z[[H']]$ module. The proof is similar to the proof of Proposition \ref{theorem4.14} and we only give an outline.

    \noindent Consider the following short exact sequence from \cite[Page-38]{ot}:
    \[
    0\lra  \dfrac{(\FF^{(p)})^\times}{((\FF^{(p)})^\times)^{p}}\lra S'(\mu_{p}/ T_\infty)\lra  Cl( T_\infty)[p] \lra 0,
    \]
    where $T_\infty \in \{F_\infty,\ \calL_\infty\}$.  

Note that $\dfrac{(\FF^{(p)})^\times}{((\FF^{(p)})^\times)^p} = 0$, so we obtain an isomorphism $S'(\mu_p/ T_\infty ) \cong Cl( T_\infty)[p]$. Similar to Proposition \ref{theorem4.14}, we apply a control theorem for the map
  $S'(\mu_p/F_\infty) \lra S'(\mu_p/\calL_\infty)^{H'}$. Then, using  the finiteness of $Cl(F_\infty)[p]$ \cite[Proposition 2]{a} together with the Nakayama lemma, we conclude that $S'(\mu_p/\calL_\infty)$ is a cofinitely generated $\Z/p\Z[[H']]$ module. 
     
    \item Next we show that the $R^S((\Z/p\Z)/ \calL_\infty)^\vee$ is a finitely generated $\Z/p\Z[[H']]$ module.  
    For this, we argue as in Lemma \ref{lemma4.11} to obtain that $R^S((\Z/pZ)/\calL_\infty)$ $\cong \Hom(Cl(\calL_\infty)/p, \Z/p\Z)$. The proof now follows from the part (1) above.
    
    \item Finally, we  consider the exact sequence for $E/K$
    \[
    0\lra \mu_p\lra E[p]\lra \Z/p\Z\lra 0.
    \]
This will give us a complex  over $\calL_\infty$, similar to the diagram \ref{diagram36} appearing in the proof of Theorem \ref{pseudonullity} . Then, from steps (1) and (2) along with a diagram chase, it is easy to see that $R^{S'}(E/\calL_\infty)^\vee/(p)$ is a finitely generated  $\Z/p\Z[[H']]$ module. By Nakayama lemma, $R^{S'}(E/\calL_\infty)^\vee$ is a finitely generated torsion $\Z_p[[H]]$ module. 
\end{enumerate}
Now for a finite set $S$ of primes of $K$ containing $S'$=$\{v_r\}$, note that $R^{S'}(E/\calL_\infty)^\vee$ surjects onto $R^{S}(E/\calL_\infty)^\vee$. Hence $R^{S}(E/\calL_\infty)^\vee$ is a pseudonull $\Z_p[[G]]$ module. 
\end{proof}

\section*{acknowledgement}
It is a pleasure to thank Taksahi Suzuki for various comments and suggestions for improvement in \S\ref{section5} of the article. { We thank the anonymous referee for various comments and suggestions which helped us to improve  the article.}  Somnath Jha acknowledges support of ANRF grant CRG/2022/005923. 

\begin{small}
	
\end{small}
\end{document}